\documentclass[11pt,a4]{article}
\usepackage{latexsym}
\usepackage{amsmath}
\usepackage{amssymb}
\usepackage{amsthm}
\usepackage{amscd}
\usepackage{color}
\definecolor{red}{rgb}{0.7,0,0}

\newtheorem{theorem}{Theorem}[section]
\newtheorem{lemma}[theorem]{Lemma}
\newtheorem{corollary}[theorem]{Corollary}
\newtheorem{proposition}[theorem]{Proposition}

\theoremstyle{definition}
\newtheorem{definition}[theorem]{Definition}

\newtheorem{remark}[theorem]{Remark}
\newtheorem{example}[theorem]{Example}

\theoremstyle{remark}

\renewcommand{\eqref}[1]{(\ref{#1})}

\renewcommand{\bigskip}{\vspace{0.2cm}}

\begin{document}

\title{Mixed Morrey spaces}

\author{Toru Nogayama \footnote{toru.nogayama@gmail.com,
Tokyo Metropolitan University,
Department of Mathematics Science,
1-1 Minami-Ohsawa, Hachioji, 192-0397, Tokyo, Japan}}

\date{}

\maketitle

\begin{abstract}
We introduce mixed Morrey spaces and show some basic properties. These properties extend the classical ones. 
We investigate the boundedness in these spaces of the iterated maximal operator, the fractional integtral operator and singular integral operator.
Furthermore, as a corollary, we obtain the boundedness of the iterated maximal operator in classical Morrey spaces.
We also establish a version of the Fefferman--Stein vector-valued maximal inequality and some weighted inequalities for the iterated maximal operator in mixed Lebesgue spaces. We point out some errors in the proof of the existing literature.
\end{abstract}

{\bf Key words}  Morrey spaces, Mixed norm, Hardy--Littlewood maximal operator, 
Fefferman--Stein vector-valued inequality, Fractional integral operator, Singular integral operator

{\bf 2010 Classification} 42B25, 42B35

\section{Introduction}
In 1961, Benedek and Panzone \cite{B-P} introduced Lebesgue spaces with mixed norm.
Bagby \cite{Bagby} showed the boundedness of the Hardy--Littlewood maximal operator for the functions taking values
in the mixed Lebesgue spaces.
Meanwhile, Morrey spaces are used to consider the boundedness of the elliptic differential operators \cite{Morrey}.
Later, many authors investigated Morrey spaces, see for example \cite{Peetre}.

In this paper, we introduce the {\it mixed Morrey space} ${\mathcal M}_{\vec{q}}^p(\mathbb{R}^n)$.
When we take a particular parameter, this space coincides with the mixed Lebesgue space $L^{\vec{p}}(\mathbb{R}^n)$ and the classical Morrey space $\mathcal{M}_q^p(\mathbb{R}^n)$. 
Our main target is the iterated maximal operator, which is obtained by repeatedly acting the one-dimentional maximal operator.
We show the boundedness of the iterated maximal operator in mixed spaces.
In particular, the boundedness in mixed Lebesgue spaces is showed by St\"{o}ckert in 1978 \cite{St}.
However, the proof is incorrect. We give a correct proof using the result of Bagby \cite{Bagby}.
Moreover, we prove some inequalities in harmonic analysis for the mixed spaces.


Throughout the paper, we use the following notation.
The letters $\vec{p}, \vec{q}, \vec{r}, \ldots$ will denote 
$n$-tuples of the numbers in $[0, \infty]$ $(n \ge 1)$, 
$\vec{p}=(p_1, \ldots, p_n), 
\vec{q}=(q_1, \ldots, q_n), \vec{r}=(r_1, \ldots, r_n)$. 
By definiton, the inequality, for example, $0<\vec{p}<\infty$ means that $0<p_i<\infty $ for each $i$.
Furthermore, for $\vec{p}=(p_1, \ldots, p_n)$ and $r \in \mathbb{R}$, let
\[
\frac{1}{\vec{p}}=
\left(\frac{1}{p_1}, \ldots, \frac{1}{p_n}\right),
\quad
\frac{\vec{p}}{r}=
\left(\frac{p_1}{r}, \ldots, \frac{p_n}{r}\right),  
\quad
\vec{p'}=(p'_1, \ldots, p'_n),
\]
where $p'_j=\frac{p_j}{p_j-1}$ is a conjugate exponent of $p_j$.
Let $Q=Q(x, r)$ be a cube having center $x$ and radius $r$, whose sides parallel to the cordinate axes. 
$|Q|$ denotes the volume of the cube $Q$ and $\ell(Q)$ denotes the side length of the cube $Q$.
By $A\lesssim B$, we denote that $A \le CB$ for some constant $C>0$, and $A \sim B$ means that 
$A\lesssim B$ and $B\lesssim A$. 

In \cite{B-P}, Benedek and Panzone introduced mixed Lebesgue spaces. 
We recall some properties and examples in Section 
\ref{sec preliminaries}.

\begin{definition}[{\it Mixed Lebesgue spaces}]{\rm \cite{B-P}}
Let $\vec{p}=(p_1, \ldots, p_n) \in (0, \infty]^n$.
Then define the {\it mixed Lebesgue norm}
$\|\cdot\|_{\vec{p}}$ or $\|\cdot\|_{(p_1,\ldots,p_n)}$ by
\begin{align*}
\|f\|_{\vec{p}}
&=
\|f\|_{(p_1,\ldots,p_n)}\\
&\equiv
\left(
\int_{{\mathbb R}} \cdots
\left(
\int_{{\mathbb R}}
\left(
\int_{{\mathbb R}}|f(x_1,x_2,\ldots,x_n)|^{p_1}{\rm d}x_1
\right)^{\frac{p_2}{p_1}}
{\rm d}x_2
\right)^{\frac{p_3}{p_2}}
\cdots{\rm d}x_n
\right)^{\frac{1}{p_n}},
\end{align*}
where $f: \mathbb{R}^n \rightarrow \mathbb{C}$ 
is a measurable function.
If $p_j=\infty$, then we have to make appropriate modifications.
We define the {\it mixed Lebesgue space}
$L^{\vec{p}}({\mathbb R}^n)=L^{(p_1,\ldots,p_n)}({\mathbb R}^n)$
to be the set of
all $f \in L^0({\mathbb R}^n)$
with
$\|f\|_{\vec{p}}<\infty$, 
where $L^0({\mathbb R}^n)$ denotes the set of measureable functions on ${\mathbb R}^n$.

\end{definition}

For all measureable functions $f$, we define the Hardy--Littlewood maximal operator $M$ by
\[
Mf(x)
=\sup_{Q \in \mathcal{Q}}
\frac{\chi_Q(x)}{|Q|}
\int_{Q}
|f(y)|
{\rm d}y,
\]
where $\mathcal{Q}$ denotes the set of all cubes in ${\mathbb R}^n$.
Let $1 \le k \le n$. Then, we define the maximal operator $M_k$ for $x_k$ as follows:
\[
M_kf(x)
\equiv
\sup_{x_k \in I}\frac{1}{|I|}\int_{I}|f(x_1, \ldots, y_k, \ldots, x_n)| {\rm d}y_k,
\]
where $I$ ranges over all intervals containing $x$.
Furthermore, for all measurable functions $f$, 
define the iterated maximal operator ${\mathcal M}_t$ by
\[
{\mathcal M}_tf(x)
\equiv
\left(
M_n \cdots M_1 \left[|f|^t\right] (x)
\right)^{\frac{1}{t}}
\]
for every $t>0$ and $x \in {\mathbb R}^n$.

We investigate the boundedness of the iterated maximal operator in mixed Lebesgue spaces.

\begin{theorem}\label{thm 171206-1} {\rm (\cite{St})}
Let $0<\vec{p}<\infty$. If $0< t<\min(p_1, \ldots, p_n)$, then
\begin{equation} \label{eq 171028-2}
\|{\mathcal M}_tf\|_{\vec{p}} \lesssim\|f\|_{\vec{p}}
\end{equation}
for $f \in L^{\vec{p}}(\mathbb{R}^n)$.
\end{theorem}
Note that this result is true but the proof is not correct in \cite{St}. We give a new proof and a counterexample of the estimate used in \cite{St} in Section \ref{sec iterated}.

Next, we define Morrey spaces.
Let $1 \le q \le p<\infty$.
Define the {\it Morrey norm}
$\|\cdot\|_{{\mathcal M}^p_q(\mathbb{R}^n)}$ by
\begin{equation*}
\| f \|_{{\mathcal M}^p_q(\mathbb{R}^n)}
\equiv
\sup\left\{
|Q|^{\frac{1}{p}-\frac{1}{q}}
\left(\int_Q|f(x)|^q\,{\rm d} x\right)^{\frac1q}
\,:\,\mbox{ $Q$ is a cube in ${\mathbb R}^n$}\right\}
\end{equation*}
for a measurable function $f$.
The {\it Morrey space}
${\mathcal M}^p_q({\mathbb R}^n)$
is the set of all measurable functions $f$
for which
$\| f \|_{{\mathcal M}^p_q(\mathbb{R}^n)}$
is finite.

Based on the above definition, we define mixed Morrey spaces, whose properties and examples will be investigated in Section 
\ref{sec mixed Morrey spaces}.
\begin{definition}[{\it Mixed Morrey spaces}]
Let $\vec{q}=(q_1, \ldots, q_n) \in (0, \infty]^n$ and $p\in (0, \infty]$ satisfy
\[
\sum_{j=1}^n\frac{1}{q_j} \ge \frac{n}{p}.
\]
Then define the {\it mixed Morrey norm}
$\|\cdot\|_{\mathcal{M}^p_{\vec{q}}(\mathbb{R}^n)}$ by
\[
\|f\|_{\mathcal{M}^p_{\vec{q}}(\mathbb{R}^n)}
\equiv
\sup\left\{
|Q|^{\frac{1}{p}-\frac{1}{n}
\left(
\sum_{j=1}^n\frac{1}{q_j}
\right)
}
\|f\chi_Q\|_{\vec{q}}
\,:\,\mbox{ $Q$ is a cube in ${\mathbb R}^n$}\right\}
\]
for $f \in L^0(\mathbb{R}^n)$.
We define the {\it mixed Morrey space} ${\mathcal{M}^p_{\vec{q}}}(\mathbb{R}^n)$ 
to be the set of all $f \in L^0(\mathbb{R}^n)$
with $\|f\|_{\mathcal{M}^p_{\vec{q}}(\mathbb{R}^n)} <\infty$.
\end{definition}

The iterated maximal operator in mixed Morrey spaces is bounded. In fact, the following holds:

\begin{theorem} \label{thm 180121-1}
Let $0<\vec{q}\le\infty$ and $0<p<\infty$ satisfy 
\[
\frac np \le\sum_{j=1}^n\frac1q_j,
\quad
\frac{n-1}{n}p<\max(q_1,\ldots,q_n).
\]
 If $0<t<\min(q_1, \ldots, q_n, p)$, then
\begin{equation*}
\|{\mathcal M}_tf\|_{\mathcal{M}^p_{\vec{q}}(\mathbb{R}^n)} 
\lesssim \|f\|_{\mathcal{M}^p_{\vec{q}}(\mathbb{R}^n)}
\end{equation*}
for all $f \in \mathcal{M}^p_{\vec{q}}(\mathbb{R}^n)$.
\end{theorem}

As a corollary, we obtain this boundedness of $\mathcal{M}_t$ in classical Morrey spaces.

\begin{corollary}
Let 
\[
0<\frac{n-1}{n}p<q\le p<\infty.
\]
 If $0<t<q$, then
\[
\|{\mathcal M}_tf\|_{\mathcal{M}^p_q(\mathbb{R}^n)} 
\lesssim \|f\|_{\mathcal{M}^p_q(\mathbb{R}^n)}
\]
for all $f \in \mathcal{M}^p_q(\mathbb{R}^n)$.
\end{corollary}

Note that Chiarenza and Frasca showed the boundedness in classical Morrey spaces 
of the Hardy--Littlewood maximal operator \cite{C-F}.
This corollary extends it. 
Furthermore, the following theorem extends Theorem \ref{thm 171206-1}.
The classical case is proved by Feffferman and Stein in 1971 \cite{F-S}.

\begin{theorem}[Dual inequality of Stein type for $L^{\vec{p}}$]\label{thm 171115-1}
Let $f$ be a measurable function on $\mathbb{R}^n$ and $w_j (j=1,\ldots,n)$ be a non-negative measurable function on $\mathbb{R}$.
Then, for $1\le\vec{p}<\infty$,
if $0< t<\min(p_1, \ldots, p_n)$ and $w_j^t\in A_{p_j}$,
\begin{eqnarray*}
\left\|
{\mathcal M}_t f\cdot \bigotimes_{j=1}^n(w_j)^{\frac{1}{p_j}}
\right\|_{\vec{p}}
\lesssim
\left\|
f \cdot \bigotimes_{j=1}^n\left(M_j w_j\right)^{\frac{1}{p_j}}
\right\|_{\vec{p}},
\end{eqnarray*}
where $\displaystyle
\left(\bigotimes_{j=1}^nw_j\right)(x)
=\prod_{j=1}^nw_j(x_j)$.
\end{theorem}

We can also extend the Feffferman--Stein vector-valued maximal inequality for mixed spaces.

\begin{theorem}
\label{thm 171123-2}
Let $0<\vec{p}<\infty$, $0<u\le\infty$ and $0<t<\min(p_1, \ldots, p_n, u)$. Then,
for every sequence $\{f_j\}_{j=1}^\infty \subset L^0(\mathbb{R}^n)$,
\begin{equation*}
\left\|\left(
\sum_{j=1}^{\infty}
[\mathcal{M}_tf_j]{}^u
\right)^{\frac{1}{u}}\right\|_{\vec{p}}
\lesssim
\left\|\left(
\sum_{j=1}^{\infty}
|f_j|^u
\right)^{\frac{1}{u}}\right\|_{\vec{p}}.
\end{equation*}

\end{theorem}

\begin{theorem} \label{thm 171219-3}
Let $1<\vec{q}<\infty$, $1<u\le\infty$, and $1<p\le\infty$ satisfy $\frac np \le\sum_{j=1}^n\frac1q_j$.
Then, for every sequence $\{f_j\}_{j=1}^{\infty} \subset L^0(\mathbb{R}^n)$,
\begin{equation*}
\left\|\left(
\sum_{j=1}^{\infty}
[Mf_j]{}^u
\right)^{\frac{1}{u}}\right\|_{\mathcal{M}^p_{\vec{q}}(\mathbb{R}^n)}
\lesssim
\left\|\left(
\sum_{j=1}^{\infty}
|f_j|^u
\right)^{\frac{1}{u}}\right\|_{\mathcal{M}^p_{\vec{q}}(\mathbb{R}^n)}.
\end{equation*}
\end{theorem}

\begin{theorem} 
Let $0<\vec{q}\le\infty$ and $0<p<\infty$ satisfy 
\[
\frac np \le\sum_{j=1}^n\frac1q_j,
\quad
\frac{n-1}{n}p<\max(q_1,\ldots,q_n).
\]
 If $0<t<\min(q_1, \ldots, q_n, u)$, then
\begin{equation*}
\left\|\left(\sum_{j=1}^\infty
[{\mathcal M}_tf_j]^u\right)^\frac1u
\right\|_{\mathcal{M}^p_{\vec{q}}(\mathbb{R}^n)} 
\lesssim \left\|\left(\sum_{k=1}^\infty|f_j|^u\right)^\frac1u\right\|_{\mathcal{M}^p_{\vec{q}}(\mathbb{R}^n)}
\end{equation*}
for $\{f_j\}_{j=1}^\infty \subset \mathcal{M}^p_{\vec{q}}(\mathbb{R}^n)$.
\end{theorem}

\begin{corollary}
Let 
\[
0<\frac{n-1}{n}p<q\le p<\infty.
\]
 If $0<t<\min(q, u)$, then
\[
\left\|\left(\sum_{j=1}^\infty
[{\mathcal M}_tf_j]^u\right)^\frac1u
\right\|_{\mathcal{M}^p_q(\mathbb{R}^n)} 
\lesssim 
\left\|
\left(\sum_{j=1}^\infty
|f_j|^u\right)^\frac1u
\right\|_{\mathcal{M}^p_q(\mathbb{R}^n)}
\]
for $\{f_j\}_{j=1}^\infty \subset \mathcal{M}^p_q(\mathbb{R}^n)$.

\end{corollary}

Furthermore, we investigate the boundedness of the fractional integral operator $I_{\alpha}$.
Its boundedness in classical Morrey spaces is proved by Adams  \cite{Adams}.
Let $0<\alpha<n$. Define the fractional integral operator $I_{\alpha}$ of order $\alpha$  by
\[
I_{\alpha}f(x)
\equiv
\int_{\mathbb{R}^n}\frac{f(y)}{|x-y|^{n-\alpha}}{\rm d}y
\]
for $f\in L^1_{\rm loc}(\mathbb{R}^n)$
as long as the right-hand side makes sense.

\begin{theorem} \label{thm 171219-1}
Let $0<\alpha<n, 1<\vec{q}, \vec{s}<\infty$ 
and $0<p, r<\infty$.
Assume that 
$\frac np \le\sum_{j=1}^n\frac1{q_j}$, and $\frac nr \le\sum_{j=1}^n\frac1{s_j}$.
Also, assume that
\[
\frac{1}{r}=\frac{1}{p}-\frac{\alpha}{n}, \quad \frac{\vec{q}}{p}=\frac{\vec{s}}{r}.
\]
Then, for $f \in \mathcal{M}^p_{\vec{q}}(\mathbb{R}^n)$,
\[
\|I_{\alpha}f\|_{\mathcal{M}^r_{\vec{s}}(\mathbb{R}^n)} \lesssim \|f\|_{\mathcal{M}^p_{\vec{q}}(\mathbb{R}^n)}. 
\]
\end{theorem}

Finally, we show that the singular integral operators are bounded in mixed Morrey spaces.
Their boundedness in classical Morrey spaces is proved by Chiarenza and Frasca \cite{C-F}.
Let $T$ be a singular integral operator with a kernel $k(x,y)$ which satisfies the following conditions:
\begin{itemize}
\item[(1)]
There exists a conctant $C>0$ such that
$|k(x,y)|\le\frac{C}{|x-y|^n}$.

\item[(2)]
There exists $\epsilon>0$ and $C>0$ such that 
\[
|k(x,y)-k(z,y)|+|k(y,x)-k(y,z)|\le C\frac{|x-z|^\epsilon}{|x-y|^{n+\epsilon}},
\]
if $|x-y|\ge2|x-z|$ with $x \neq y$.

\item[(3)]
If $f \in L_c^\infty(\mathbb{R}^n)$, the set of all compactly supported $L^\infty$-functions, then
\[
Tf(x)=\int_{\mathbb{R}^n}k(x,y)f(y){\rm d}y  \quad  (x \notin {\rm supp}(f)).
\]

\end{itemize}

Keeping in mind that $T$ extends to a bounded linear operator on $\mathcal{M}^p_{q}(\mathbb{R}^n)$,
we prove the following theorem.

\begin{theorem} \label{thm 180514-1}
Let $1<\vec{q}<\infty$ and $1<p<\infty$ satysfy 
\[
\frac np \le\sum_{j=1}^n\frac1{q_j}.
\]
Then,
\begin{equation*}
\|Tf\|_{\mathcal{M}_{\vec{q}}^p(\mathbb{R}^n)} \lesssim \|f\|_{\mathcal{M}_{\vec{q}}^p(\mathbb{R}^n)} 
\end{equation*}
for $f \in \mathcal{M}_{\vec{q}}^p(\mathbb{R}^n)$.
\end{theorem}

We organize the remaining part of this paper as follows:
In Sections \ref{sec preliminaries} and \ref{sec mixed Morrey spaces}, we investigate some properties 
and present examples of mixed Lebesgue spaces and mixed Morrey spaces, respectively.
We prove the boundedness of the iterated maximal operator in mixed spaces in Section \ref{sec iterated}. 
In Section \ref{sec dual inequality}, we show the dual inequality of Stein type for mixed Lebesgue spaces. 
Section \ref{sec vector-valued} is devoted to the vector-valued extension of Section \ref{sec iterated}. 
Finally, we prove that the fractional integral operator and singular integral operator are bounded in mixed Morrey spaces in Section \ref{sec fractional}.

\section{Preliminaries} \label{sec preliminaries}
\subsection{Mixed Lebesgue spaces}
In this subsection, we recall  the mixed Lebesgue space $L^{\vec{p}}({\mathbb R}^n)$
which is  introduced by Benedek and Panzone in \cite{B-P}.
This space has properties similar to classical Lebesgue space. 
First, $L^{\vec{p}}({\mathbb R}^n)$ is a Banach space for $1 \le \vec{p} \le\infty$.
H\"older's inequality holds:
Let $1<\vec{p}, \vec{q}<\infty$ and define $\vec{r}$ so that
$\frac{1}{\vec{p}}+\frac{1}{\vec{q}}=\frac{1}{\vec{r}}$.
If $f \in L^{\vec{p}}({\mathbb R}^n), g \in L^{\vec{q}}({\mathbb R}^n)$, then $fg \in L^{\vec{r}}({\mathbb R}^n)$, and
$\|fg\|_{\vec{r}} \le \|f\|_{\vec{p}}\|g\|_{\vec{q}}$.
Furthermore, the monotone convergence theorem, Fatou's lemma and the Lebesgue convergence theorem also follow.


\begin{remark} \label{rem 171228-1}
Let $\vec{p}\in(0,\infty]^n$ and $f$ be a measureable function on ${\mathbb R}^n$.
\begin{itemize}
\item[(i)]
If for each $p_i = p$, then
\[
\|f\|_{\vec{p}}
=
\|f\|_{(p_1,\ldots,p_n)}
=
\left(
\int_{\mathbb{R}^n}|f(x)|^p
{\rm d}x
\right)^{\frac{1}{p}}
=\|f\|_{p}
\]
and
\[
L^{\vec{p}}({\mathbb R}^n)=L^p({\mathbb R}^n).
\]

\item[(ii)]

For any $(x_2,\ldots,x_n) \in \mathbb{R}^{n-1}$, 
\[
\|f\|_{(p_1)}(x_2, \ldots,x_n) \equiv \left(
\int_{\mathbb{R}}|f(x_1, \ldots, x_n)|^{p_1}
{\rm d}x_1
\right)^{\frac{1}{p_1}}
\]
is a measurable function and defined on  ${\mathbb R}^{n-1}$. 
Moreover, we define
\[
\|f\|_{\vec{q}}= \|f\|_{(p_1,\ldots,p_j)} \equiv \|[\|f\|_{(p_1,\ldots,p_{j-1})}]\|_{(p_j)},
\]
where $\|f\|_{(p_1,\ldots,p_{j-1})}$ denotes $|f|$, if $j=1$
and $\vec{q}=(p_1, \ldots, p_j), j\leq n$.
Note that $\|f\|_{\vec{q}}$ is a measurable function of $(x_{j+1},\ldots,x_n)$ for $j<n$.


\end{itemize}
\end{remark}

Next, we consider the examples of $L^{\vec{p}}({\mathbb R}^n)$.

\begin{example} \label{ex 171109-1}
Let $f_1 \ldots, f_n \in L^0(\mathbb{R}) \setminus \{0\}$.
Then $f=\bigotimes_{j=1}^nf_j \in L^{\vec{p}}(\mathbb{R}^n)$ if and only if $f_j \in L^{p_j}({\mathbb R})$
for each $j=1, \ldots, n$. 
In fact,  
\begin{align*}
\|f\|_{\vec{p}}
&=
\left(
\int_{{\mathbb R}} \cdots
\left(
\int_{{\mathbb R}}
\left(
\int_{{\mathbb R}}\prod_{j=1}^n|f_j(x_j)|^{p_1}{\rm d}x_1
\right)^{\frac{p_2}{p_1}}
{\rm d}x_2
\right)^{\frac{p_3}{p_2}}
\cdots{\rm d}x_n
\right)^{\frac{1}{p_n}}\\
&=
\prod_{j=1}^n
\left(
\int_{\mathbb{R}}
|f_j(x_j)|^{p_j}
{\rm d}x_j
\right)^{\frac{1}{p_j}}
=
\prod_{j=1}^n
\|f_j\|_{p_j}.
\end{align*}

\end{example}

\begin{example}\label{ex 171108-1}
Let $Q$ be a cube. Then, for $0<\vec{p}\le\infty$,
\[
\|\chi_Q\|_{\vec{p}}=|Q|^{\frac{1}{n}(\frac{1}{p_1}+\cdots+ \frac{1}{p_n})}.
\]

In fact, we can write $Q=I_1 \times \cdots \times I_n$,
where each $I_j$ is an interval of equal length. Hence, $\chi_Q(x)=\prod_{j=1}^n\chi_{I_j}(x_j)$.
Using Example \ref{ex 171109-1}, we have
\begin{align*}
\|\chi_Q\|_{\vec{p}}
=
\prod_{j=1}^n
\|\chi_{I_j}\|_{p_j}
=
\prod_{j=1}^n
\left(
\int_{I_j} 
{\rm d}x_j
\right)^{\frac{1}{p_j}}
=
\prod_{j=1}^n
|I_j|^{\frac{1}{p_j}}.
\end{align*}
Notice that since $Q$ is a cube,
$|I_j|=\ell(Q)=|Q|^{\frac{1}{n}}.$
Thus, 
\[
\|\chi_Q\|_{\vec{p}}= \prod_{j=1}^n|I_j|^{\frac{1}{p_j}}
=|Q|^{\frac{1}{n}(\frac{1}{p_1}+\cdots+ \frac{1}{p_n})}.
\]
\end{example}

\begin{example}\label{ex 180112-1}
Let $m=(m_1,\ldots, m_n) \in \mathbb{Z}^n$ and $\{a_m\}_{m \in \mathbb{Z}^n} \subset \mathbb{C}$. Define
\[
f(x)=\sum_{m \in \mathbb{Z}^n}a_m\chi_{m+[0,1]^n}(x).
\]
Then, 
\[
\|f\|_{\vec{p}}
=
\left(\sum_{m_n \in \mathbb{Z}}\cdots\left(\sum_{m_1 \in \mathbb{Z}}
\left|a_{(m_1,\ldots, m_n)}\right|^{p_1}
\right)^{\frac{p_2}{p_1}}
\cdots\right)^{\frac{1}{p_n}}.
\]
In fact, 
\begin{align*}
\|f\|_{\vec{p}}
&=
\left(\int_{\mathbb{R}}\cdots\left(\int_{\mathbb{R}}
\left|\sum_{m \in \mathbb{Z}^n}a_m\chi_{m+[0,1]^n}(x)\right|^{p_1}
{\rm d}x_1\right)^{\frac{p_2}{p_1}}
\cdots{\rm d}x_n\right)^{\frac{1}{p_n}}\\
&=
\left(\sum_{m_n \in \mathbb{Z}}\int_{m_n}^{m_n+1}\cdots\left(\sum_{m_1 \in \mathbb{Z}}\int_{m_1}^{m_1+1}
\left|\sum_{m \in \mathbb{Z}^n}a_m\chi_{m+[0,1]^n}(x)\right|^{p_1}
{\rm d}x_1\right)^{\frac{p_2}{p_1}}
\cdots{\rm d}x_n\right)^{\frac{1}{p_n}}\\
&=
\left(\sum_{m_n \in \mathbb{Z}}\int_{m_n}^{m_n+1}\cdots\left(\sum_{m_1 \in \mathbb{Z}}\int_{m_1}^{m_1+1}
\left|a_{(m_1,\ldots, m_n)}\right|^{p_1}
{\rm d}x_1\right)^{\frac{p_2}{p_1}}
\cdots{\rm d}x_n\right)^{\frac{1}{p_n}}\\
&=
\left(\sum_{m_n \in \mathbb{Z}}\cdots\left(\sum_{m_1 \in \mathbb{Z}}
\left|a_{(m_1,\ldots, m_n)}\right|^{p_1}
\right)^{\frac{p_2}{p_1}}
\cdots\right)^{\frac{1}{p_n}}.\\
\end{align*}
\end{example}
We can consider the last term as a mixed sequence norm, which
computes respectively $\ell^{p_i}$-norm with respect to $m_i$.
We denote it by $\|\{a_m\}_{m \in \mathbb{Z}^n}\|_{\ell^{(p_1, \ldots, p_n)}}$:
\begin{align*}
\|\{a_m\}_{m \in \mathbb{Z}^n}\|_{\ell^{(p_1, \ldots, p_n)}}\nonumber
&=\|a_{(m_1, \ldots, m_n)}\|_{\ell^{(p_1, \ldots, p_n)}}\\ 
&\equiv
\left(\sum_{m_n \in \mathbb{Z}}\cdots\left(\sum_{m_2 \in \mathbb{Z}}\left(\sum_{m_1 \in \mathbb{Z}}
\left|a_{(m_1,\ldots, m_n)}\right|^{p_1}
\right)^{\frac{p_2}{p_1}}
\right)^{\frac{p_3}{p_2}}
\cdots\right)^{\frac{1}{p_n}}.
\end{align*}
Furthermore, this norm is also defined inductively:
\[
\|a_{(m_1, \ldots, m_n)}\|_{\ell^{(p_1, \ldots, p_j)}}
\equiv\left\|\left[\|a_{(m_1, \ldots, m_n)}\|_{\ell^{(p_1, \ldots, p_{j-1})}}\right]\right\|_{\ell^{(p_j)}},
\]
where $\|a_{(m_1, \ldots, m_n)}\|_{\ell^{(p_1, \ldots, p_{j-1})}}=|a_{(m_1, \ldots, m_n)}|$ if $j=1$ and
\[
\|a_{(m_1, \ldots, m_n)}\|_{\ell^{(p_j)}}\equiv\left(\sum_{m_j\in\mathbb{Z}}|a_{(m_1, \ldots, m_n)}|^{p_j}\right)^\frac{1}{p_j}
\]
for $j=1,\ldots, n$.

Next, we consider the properties of mixed Lebesgue spaces.
Since these proofs are elementary, we omit the detail.

\begin{proposition}\label{prop 180113-2}
Let $0<\vec{p}\le\infty$.
The mixed Lebesgue norm has the dilation relation: for all $f \in L^{\vec{p}}(\mathbb{R}^n)$ and $t>0$,
\begin{equation} \label{eq 171113-1}
\|f(t\cdot)\|_{\vec{p}}=t^{-\sum_{j=1}^n\frac{1}{p_j}}\|f\|_{\vec{p}}.
\end{equation}
\end{proposition}

\begin{proposition}[Fatou's property for $L^{\vec{p}}(\mathbb{R}^n)$]
\label{prop 171225-1}
Let $0<\vec{p}\le\infty$.
Let $\{f_j\}_{j=1}^\infty$ be a sequence of non-negative measurable functions on $\mathbb{R}^n$.
Then, 
\[
\left\|\varliminf_{j \to \infty}f_j\right\|_{\vec{p}}\le \varliminf_{j \to \infty}\|f_j\|_{\vec{p}}.
\]
\end{proposition}


\subsection{$A_p$ weights and extrapolation}
By a weight we mean a measurable function
which satisfies $0<w(x)<\infty$
for almost all $x \in {\mathbb R}^n$.

\begin{definition}
Let $1<p<\infty$ and $w$ be a weight.
Then, $w$ is said to be an $A_p$ weight if
\[
[w]_{A_p}
=\sup_{Q \in \mathcal{Q}} \left(\frac{1}{|Q|}\int_Qw(x){\rm d}x\right)
\left(\frac{1}{|Q|}\int_Qw(x)^{-\frac{1}{p-1}}{\rm d}x\right)^{p-1}<\infty.
\]
A weight $w$ is said to be an $A_1$ weight if
\[
[w]_{A_1}
=\sup_{Q \in \mathcal{Q}} \left(\frac{1}{|Q|}\int_Qw(x){\rm d}x\right)
{\rm ess}\sup_{x \in Q}w(x)^{-1}<\infty.
\]
\end{definition}

Let $0<p<\infty$, and let $w$ be a weight.
One defines
\[
\|f\|_{L^p(w)}
\equiv
\left(
\int_{{\mathbb R}^n}|f(x)|^p w(x){\rm d}x
\right)^{\frac1p} \quad (f \in L^0(\mathbb{R}^n)).
\]
The space $L^p(w)$ is the set of all
measurable functions $f$
for which the norm
$\|f\|_{L^p(w)}$ is finite.
The space $L^p(w)$
is called the {\it weighted Lebesgue space}
or the $L^p${\it-space with weight} $w$.

When we consider estimates of $A_p$-weights,
we face the following type of estimate:
\begin{equation}\label{eq:170318-1}
\|T h\|_{L^p(W)}
\le N([W]_{A_{p}})
\|h\|_{L^p(W)}
\quad (h \in L^p(W)),
\end{equation}
where $T$ is a mapping from $L^p(W)$ to $L^0({\mathbb R}^n)$
and $N$ is a positive increasing function defined on $[1,\infty)$.
Extrapolation is a technique
to expand the validity
of (\ref{eq:170318-1})
for all $1<p<\infty$
based on the validity of $(\ref{eq:170318-1})$
for some $p_0$.
We invoke the following extrapolation result from \cite{CGMP06}:
\begin{proposition}\label{thm:161125-1}
Let $N=N(\cdot):[1,\infty) \to [1,\infty)$
be an increasing function,
and let $1<p_0,p<\infty$.
Suppose that we have a family
${\mathcal F}$ of the couple of measurable functions
$(f,g)$ satisfying
\begin{equation*}
\|f\|_{L^{p_0}(W)}
\le N([W]_{A_{p_0}})
\|g\|_{L^{p_0}(W)}
\end{equation*}
for all $(f,g) \in {\mathcal F}$ and $W \in A_{p_0}$.

Then 
\begin{equation*}
\|f\|_{L^{p}(w)}
\lesssim_{[w]_{A_p}}
\|g\|_{L^{p}(w)}
\end{equation*}
for all $(f,g) \in {\mathcal F}$ and $w \in A_{p}$.
\end{proposition}

\section{Mixed Morrey spaces} \label{sec mixed Morrey spaces}

In this section, we discuss some properties and examples of mixed Morrey spaces.
We recall the definition of mixed Morrey spaces.
Let $0<\vec{q}\le \infty$, $0<p\le\infty$ satisfy
\[
\sum_{j=1}^n\frac{1}{q_j} \ge \frac{n}{p}.
\]
Then define the {\it mixed Morrey norm}
$\|\cdot\|_{\mathcal{M}^p_{\vec{q}}(\mathbb{R}^n)}$ by
\[
\|f\|_{\mathcal{M}^p_{\vec{q}}(\mathbb{R}^n)}
\equiv
\sup\left\{
|Q|^{\frac{1}{p}-\frac{1}{n}
\left(
\sum_{j=1}^n\frac{1}{q_j}
\right)
}
\|f\chi_Q\|_{\vec{q}}
\,:\,\mbox{ $Q$ is a cube in ${\mathbb R}^n$}\right\}.
\]
We define the {\it mixed Morrey space} ${\mathcal{M}^p_{\vec{q}}}(\mathbb{R}^n)$ 
to be the set of all $f \in L^0(\mathbb{R}^n)$
with $\|f\|_{\mathcal{M}^p_{\vec{q}}(\mathbb{R}^n)} <\infty$.

\begin{remark}
Let $\vec{q}\in (0, \infty]^n$ and $f \in L^0(\mathbb{R}^n)$.
\begin{itemize}

\item[(i)]
If for each $q_i=q$, then by Remark \ref{rem 171228-1}
\[
|Q|^{\frac{1}{p}-\frac{1}{n}
\left(
\sum_{j=1}^n\frac{1}{q_j}
\right)
}
\|f\chi_Q\|_{\vec{q}}
=|Q|^{\frac{1}{p}-\frac{1}{n}
\left(
\sum_{j=1}^n\frac{1}{q}
\right)
}
\|f\chi_Q\|_{\vec{q}}
=|Q|^{\frac{1}{p}-\frac{1}{q}}
\|f\chi_Q\|_{q}.
\]
Thus, taking the supremum over the all cubes in $\mathbb{R}^n$, we obtain
\[
\|f\|_{{\mathcal{M}^p_{\vec{q}}}(\mathbb{R}^n)}
=
\|f\|_{{\mathcal{M}^p_q}(\mathbb{R}^n)},
\]
and
\[
\mathcal{M}^p_{\vec{q}}(\mathbb{R}^n)
=
\mathcal{M}^p_q(\mathbb{R}^n),
\]
with coincidence of norms.
\item[(ii)]
In particular, let
\[
\hspace{10pt}
p=\frac{n}{1/q_1+ \cdots +1/q_n}.
\]
Then, since
\begin{align*}
\|f\|_{\mathcal{M}^p_{\vec{q}}(\mathbb{R}^n)}
&=
\sup\left\{
|Q|^{\frac{1}{p}-\frac{1}{n}
\left(
\sum_{j=1}^n\frac{1}{q_j}
\right)
}
\|f\chi_Q\|_{\vec{q}}
\,:\,\mbox{ $Q$ is a cube in ${\mathbb R}^n$}\right\}\\
&=
\sup\left\{
\|f\chi_Q\|_{\vec{q}}
\,:\,\mbox{ $Q$ is a cube in ${\mathbb R}^n$}\right\}
=\|f\|_{\vec{q}},
\end{align*}
we obtain
\[
L^{\vec{q}}(\mathbb{R}^n)
={\mathcal{M}^p_{\vec{q}}}(\mathbb{R}^n),
\]
with coincidence of norms.

\item[(iii)]
The mixed Morrey space ${\mathcal{M}^p_{\vec{q}}}(\mathbb{R}^n)$ is also a Banach space for $1\le \vec{q}\le\infty$ and $0<p\le\infty$.
Although the proof is easy, we give the proof for the sake of completeness.
First, we will check the triangle inequality.
For $f, g \in {\mathcal{M}^p_{\vec{q}}}({\mathbb R}^n)$, 
\begin{align*}
\|f+g\|_{{\mathcal{M}^p_{\vec{q}}}({\mathbb R}^n)}
&=
\sup_{Q}|Q|^{\frac{1}{p}-\frac{1}{n}
\left(\sum_{j=1}^n\frac{1}{q_j}\right)}
\left\|(f+g)\chi_Q\right\|_{\vec{q}}\\
&\le
\sup_{Q}|Q|^{\frac{1}{p}-\frac{1}{n}
\left(\sum_{j=1}^n\frac{1}{q_j}\right)}
\left(\left\|f\chi_Q\right\|_{\vec{q}}+\left\|g\chi_Q\right\|_{\vec{q}}\right)\\
&\le
\|f\|_{\mathcal{M}^p_{\vec{q}}(\mathbb{R}^n)}+\|g\|_{\mathcal{M}^p_{\vec{q}}(\mathbb{R}^n)}.
\end{align*}
The positivity and the homogeneity are both clear.
Thus, $\mathcal{M}^p_{\vec{q}}({\mathbb R}^n)$ is a normed space.
It remains to check the completeness.

Let $\{f_j\}_{j=1}^{\infty} \subset \mathcal{M}^p_{\vec{q}}({\mathbb R}^n)$
and $\sum_{j=1}^{\infty} \|f_j\|_{\mathcal{M}^p_{\vec{q}}({\mathbb R}^n)} <\infty$. Then,
\[
\left
\|\sum_{j=1}^{J} |f_j|
\right\|_{\mathcal{M}^p_{\vec{q}}({\mathbb R}^n)} 
\le \sum_{j=1}^{J} \|f_j\|_{\mathcal{M}^p_{\vec{q}}({\mathbb R}^n)} 
\le \sum_{j=1}^{\infty} \|f_j\|_{\mathcal{M}^p_{\vec{q}}({\mathbb R}^n)}<\infty.
\]
By Proposition \ref{prop 171225-1},
\begin{align*}
\left
\|\sum_{j=1}^{\infty} |f_j|\right\|_{\mathcal{M}^p_{\vec{q}}({\mathbb R}^n)} 
&=
\left\|\lim_{J \to \infty} \sum_{j=1}^{J} |f_j|\right\|_{\mathcal{M}^p_{\vec{q}}({\mathbb R}^n)} 
\le
\lim_{J \to \infty} \left\|\sum_{j=1}^{J} |f_j|\right\|_{\mathcal{M}^p_{\vec{q}}({\mathbb R}^n)} \\
&\le
\lim_{J \to \infty} \sum_{j=1}^{J} \left\|f_j\right\|_{\mathcal{M}^p_{\vec{q}}({\mathbb R}^n)} 
=
\sum_{j=1}^{\infty} \|f_j\|_{\mathcal{M}^p_{\vec{q}}({\mathbb R}^n)} 
<\infty.
\end{align*}
Thus, for almost everywhere $x \in {\mathbb R}^n$,
\[
\sum_{j=1}^{\infty} |f_j(x)| <\infty. 
\]
Therefore, there exists a function $g$ such that the limit
\[
g(x) \equiv \lim_{J \to \infty} \sum_{j=1}^{J} f_j(x)
\]
exists for almost everywhere $x \in {\mathbb R}^n$.
If $\sum_{j=1}^{\infty} |f_j(x)|=\infty$, then it will be understood that $g(x)\equiv0$.
Again, by Proposition \ref{prop 171225-1}, for $m>1$
\begin{align*}
\left
\|g-\sum_{j=1}^{m-1} f_j
\right\|_{\mathcal{M}^p_{\vec{q}}({\mathbb R}^n)}
&=\left\|
\sum_{j=m}^{\infty} f_j
\right\|_{\mathcal{M}^p_{\vec{q}}({\mathbb R}^n)}
= 
\left\|\lim_{J \to \infty} \sum_{j=m}^{J} f_j
\right\|_{\mathcal{M}^p_{\vec{q}}({\mathbb R}^n)}\\ 
&\le
\lim_{J \to \infty} \left\|\sum_{j=m}^{J} f_j
\right\|_{\mathcal{M}^p_{\vec{q}}({\mathbb R}^n)} 
\le
\lim_{J \to \infty} \sum_{j=m}^{J} \left\|f_j
\right\|_{\mathcal{M}^p_{\vec{q}}({\mathbb R}^n)} \\
&=
\sum_{j=m}^{\infty} \|f_j\|_{\mathcal{M}^p_{\vec{q}}({\mathbb R}^n)} .
\end{align*}
Letting $m \to \infty$, we obtain
\[
g=\sum_{j=1}^{\infty} f_j
\]
in $\mathcal{M}^p_{\vec{q}}({\mathbb R}^n)$.

\end{itemize}
\end{remark}

First, we give the properties of the {\it mixed Morrey spaces.}

\begin{proposition} \label{prop 171120-2} 
Let $\vec{q}\in (0, \infty]^n$ and $p\in (0, \infty]$.
The mixed Morrey norm has the following dilation relation: for all $f \in L^0(\mathbb{R}^n)$ and $t>0$,
\begin{equation} \label{eq 171118-2}
\|f(t\cdot)\|_{\mathcal{M}^p_{\vec{q}}(\mathbb{R}^n)}
=t^{-\frac{n}{p}}\|f\|_{\mathcal{M}^p_{\vec{q}}(\mathbb{R}^n)}.
\end{equation}
\end{proposition}
\begin{proof}
Although the proof is elementary again, we give the proof for the sake of completeness.
To see (\ref{eq 171118-2}), using (\ref{eq 171113-1}), we obtain
\begin{align*}
\|f(t\cdot)\|_{\mathcal{M}^p_{\vec{q}}(\mathbb{R}^n)}
&=
\sup_{Q=Q(x,r)}
|Q(x,r)|^{\frac{1}{p}-\frac{1}{n}
\left(
\sum_{j=1}^n\frac{1}{q_j}
\right)
}
\|f(t\cdot)\chi_{Q(x,r)}\|_{\vec{q}}\\
&=
\sup_{Q=Q(x,r)}
|Q(x,r)|^{\frac{1}{p}-\frac{1}{n}
\left(
\sum_{j=1}^n\frac{1}{q_j}
\right)
}
t^{-\sum_{j=1}^n\frac{1}{q_j}}
\|f\chi_{Q(tx,tr)}\|_{\vec{q}}\\
&=
\sup_{Q=Q(x,r)}
|Q(tx,tr)|^{\frac{1}{p}-\frac{1}{n}
\left(
\sum_{j=1}^n\frac{1}{q_j}
\right)
}
t^{-\frac{n}{p}}
\|f\chi_{Q(tx,tr)}\|_{\vec{q}}\\
&=
t^{-\frac{n}{p}}\|f\|_{\mathcal{M}^p_{\vec{q}}(\mathbb{R}^n)}.
\end{align*}
\end{proof}
\begin{proposition}\label{prop 180521-1}

Let $0<\vec{q} \le \vec{r} \le \infty$, $0<p<\infty$, and
assume $\frac{1}{r_1}+\cdots+\frac{1}{r_n} \ge \frac{n}{p}$.
Then, 
\[
{\mathcal{M}^p_{\vec{r}}}(\mathbb{R}^n) 
\subset
{\mathcal{M}^p_{\vec{q}}}(\mathbb{R}^n).
\]
\end{proposition}

\begin{proof}
To get this inclusion, it suffices to show that for all $f \in L^0{(\mathbb{R}^n)}$
and all cubes $Q$,
\begin{equation} \label{eq 171029-1}
|Q|^{\frac{1}{p}-\frac{1}{n}
\left(
\sum_{j=1}^n\frac{1}{q_j}
\right)
}
\|f\chi_Q\|_{\vec{q}}
\le
|Q|^{\frac{1}{p}-\frac{1}{n}
\left(
\sum_{j=1}^n\frac{1}{r_j}
\right)
}
\|f\chi_Q\|_{\vec{r}}.
\end{equation}
Once we can show (\ref{eq 171029-1}),
taking the supremum over the all cubes in $\mathbb{R}^n$,
we have
\[
\|f\|_{\mathcal{M}^p_{\vec{q}}(\mathbb{R}^n)}
\le
\|f\|_{\mathcal{M}^p_{\vec{r}}(\mathbb{R}^n)}.
\]
This implies that
\[
{\mathcal{M}^p_{\vec{r}}}(\mathbb{R}^n) 
\subset
{\mathcal{M}^p_{\vec{q}}}(\mathbb{R}^n).
\]
So we shall show (\ref{eq 171029-1}).
Note that we can write $Q=I_1 \times \cdots \times I_n$, where each $I_j$ is an interval of equal length.
Using H\"{o}lder's inequality, we have
\begin{align*} 
&\|f\chi_Q\|_{\vec{q}}\\
&=\left(
\int_{I_n} \cdots
\left(
\int_{I_2}
\left(
\int_{I_1}|f(x)|^{q_1}{\rm d}x_1
\right)^{\frac{q_2}{q_1}}
{\rm d}x_2
\right)^{\frac{q_3}{q_2}}
\cdots{\rm d}x_n
\right)^{\frac{1}{q_n}}\\
&\le
\left(
\int_{I_n} \cdots
\left(
\int_{I_2}
\left[
\left(
\int_{I_1}|f(x)|^{{q_1}\frac{r_1}{q_1}}{\rm d}x_1
\right)^{\frac{q_1}{r_1}}
\left(
\int_{I_1}{\rm d}x_1
\right)^{1-\frac{q_1}{r_1}}
\right]^{\frac{q_2}{q_1}}
{\rm d}x_2
\right)^{\frac{q_3}{q_2}}
\cdots{\rm d}x_n
\right)^{\frac{1}{q_n}}\\
&=
\left(
\int_{I_n} \cdots
\left(
\int_{I_2}
\|f\chi_{I_1 \times {\mathbb R}^{n-1}}\|_{(r_1)}(x_2, \ldots, x_n)^{q_2}
|I_1|^{\frac{q_2}{q_1}-\frac{q_2}{r_1}}
{\rm d}x_2
\right)^{\frac{q_3}{q_2}}
\cdots{\rm d}x_n
\right)^{\frac{1}{q_n}}.\\
\end{align*}
Since $|I_1|=\ell(Q)$, 
\begin{align*}
\|f\chi_Q\|_{\vec{q}}
&\le
\left(
\int_{I_n} \cdots
\left(
\int_{I_2}
\|f\chi_{I_1 \times {\mathbb R}^{n-1}}\|_{(r_1)}(x_2, \ldots, x_n)^{q_2}
\ell(Q)
^{\frac{q_2}{q_1}-\frac{q_2}{r_1}}
{\rm d}x_2
\right)^{\frac{q_3}{q_2}}
\cdots{\rm d}x_n
\right)^{\frac{1}{q_n}}\\
&=
\ell(Q)
^{\frac{1}{q_1}-\frac{1}{r_1}}
\left(
\int_{I_n} \cdots
\left(
\int_{I_2}
\|f\chi_{I_1 \times {\mathbb R}^{n-1}}\|_{(r_1)}(x_2, \ldots, x_n)^{q_2}
{\rm d}x_2
\right)^{\frac{q_3}{q_2}}
\cdots{\rm d}x_n
\right)^{\frac{1}{q_n}}.\\
\end{align*}
Iterating this procedure, we get
\begin{align*}
\|f\chi_Q\|_{\vec{q}}
&\le
\ell(Q)
^{\left(
\sum_{j=1}^{n-1}\frac{1}{q_j}
\right)
-\left(
\sum_{j=1}^{n-1}\frac{1}{r_j}
\right)}
\left(
\int_{I_n}
\|f\chi_{I_1 \times \cdots \times I_{n-1} \times {\mathbb R}}\|_{(r_1, \cdots, r_{n-1})}(x_n)^{q_n}
{\rm d}x_n
\right)^{\frac{1}{q_n}}\\
&\le
\ell(Q)
^{\left(
\sum_{j=1}^n\frac{1}{q_j}
\right)
-\left(
\sum_{j=1}^n\frac{1}{r_j}
\right)}
\|f\chi_{Q}\|_{\vec{r}}.
\end{align*}
Thus, we obtain
\[
|Q|^{\frac{1}{p}-\frac{1}{n}
\left(
\sum_{j=1}^n\frac{1}{q_j}
\right)
}
\|f\chi_Q\|_{\vec{q}}
\le
|Q|^{\frac{1}{p}-\frac{1}{n}
\left(
\sum_{j=1}^n\frac{1}{r_j}
\right)
}
\|f\chi_Q\|_{\vec{r}}.
\]

\end{proof}
Let us give some examples.

\begin{example} \label{ex 171120-6}
In the classical case, it is known that $f(x)=|x|^{-\frac{n}{p}} \in \mathcal{M}^p_q(\mathbb{R}^n)$ if $q<p$.
Let $\vec{q}=(q_1, \ldots, q_n)$. Using the above embedding, we have
\[
\mathcal{M}^p_{\tilde{q}}(\mathbb{R}^n)
=
\mathcal{M}^p_{(\underbrace{\tilde{q}, \ldots, \tilde{q}}_{\mbox{$n$ times}})}(\mathbb{R}^n)
\subset
\mathcal{M}^p_{\vec{q}}(\mathbb{R}^n),
\]
where $\tilde{q}=\max(q_1, \ldots, q_n).$
 Thus, if $\max(q_1, \ldots, q_n)=\tilde{q}<p$,
\[
f(x)=|x|^{-\frac{n}{p}} \in \mathcal{M}^p_{\vec{q}}(\mathbb{R}^n).
\]
\end{example}

\begin{remark}
In Example \ref{ex 171120-6}, the condition  
\begin{equation} \label{eq 171120-7}
\max(q_1, \ldots, q_n)=\tilde{q}<p
\end{equation}
is a sufficient condition but is not a necessary condition for
$f(x)=|x|^{-\frac{n}{p}} \in \mathcal{M}^p_{\vec{q}}(\mathbb{R}^n)$. 
In fact, let $\vec{s}=(s_1, \underbrace{\infty, \ldots, \infty}_{\mbox{$(n-1)$ times}})$ and $s_1<\frac{q_1}{n}$.
Then, by Proposition \ref{prop 171120-2},
\begin{align*}
\|f\|_{\mathcal{M}^p_{\vec{s}}(\mathbb{R}^n)}
&=
\sup_{Q=Q(x,r)}
|Q(x,r)|^{\frac{1}{p}-\frac{1}{n}
\left(
\sum_{j=1}^n\frac{1}{q_j}
\right)
}
\|f\chi_{Q(x,r)}\|_{\vec{s}}\\
&=
\sup_{r>0}
|Q(0,r)|^{\frac{1}{p}-\frac{1}{n}
\left(
\sum_{j=1}^n\frac{1}{q_j}
\right)
}
\|f\chi_{Q(0,r)}\|_{\vec{s}}\\
&=
|Q(0,1)|^{\frac{1}{p}-\frac{1}{n}
\left(
\sum_{j=1}^n\frac{1}{q_j}
\right)
}
\|f\chi_{Q(0,1)}\|_{\vec{s}}\\
&=
|Q(0,1)|^{\frac{1}{p}-\frac{1}{n}\left(\sum_{j=1}^n\frac{1}{q_j}\right)}
\left\|
\left(
\int_{-1}^1
|x|^{-\frac{n}{p}s_1}
{\rm d}x_1
\right)^{\frac{1}{s_1}}\chi_{[-1, 1]^{n-1}}
\right\|_{(\underbrace{\infty, \ldots, \infty}_{\mbox{$(n-1)$ times}})}.
\end{align*}
Since $s_1<\frac{q_1}{n}$, 
$\|f\|_{\mathcal{M}^p_{\vec{s}}(\mathbb{R}^n)}<\infty$
and $f \in \mathcal{M}^p_{\vec{s}}(\mathbb{R}^n)$.
But $\vec{s}$ does not satisfy (\ref{eq 171120-7}). 
\end{remark}

\begin{example} \label{ex 171120-1}
Let $0<\vec{q}\le\infty$ and assume that
$q_j<p_j$ if $p_j<\infty$
and that 
$q_j\le \infty$ if $p_j=\infty$
 $(j=1, \ldots, n)$.
Let
\begin{equation} \label{eq 171120-5}
\sum_{j=1}^n\frac{1}{p_j}=\frac{n}{p}.
\end{equation}
Then,
\[
f(x)=\prod_{j=1}^n |x_j|^{-\frac{1}{p_j}} \in \mathcal{M}^p_{\vec{q}}(\mathbb{R}^n).
\]

In fact, letting $Q=I_1 \times \cdots \times I_n$, we obtain
\begin{align*}
\|f\chi_Q\|_{\vec{q}} 
&=
\left(
\int_{I_n} \cdots
\left(
\int_{I_2}
\left(
\int_{I_1}
\prod_{j=1}^n |x_j|^{-\frac{q_1}{p_j}}
{\rm d}x_1
\right)^{\frac{q_2}{q_1}}
{\rm d}x_2
\right)^{\frac{q_3}{q_2}}
\cdots{\rm d}x_n
\right)^{\frac{1}{q_n}}\\
&=
\prod_{j=1}^n
\left(
\int_{I_j}
|x_j|^{-\frac{q_j}{p_j}}
{\rm d}x_j
\right)^{\frac{1}{q_j}}.
\end{align*}
To estimate this integral, letting $\ell(Q)=r$, we have
\begin{align*}
\int_{I_j}
|x_j|^{-\frac{q_j}{p_j}}
{\rm d}x_j
&\le
\int_{-r/2}^{r/2}
|x_j|^{-\frac{q_j}{p_j}}
{\rm d}x_j
=
2\int_{0}^{r/2}
|x_j|^{-\frac{q_j}{p_j}}
{\rm d}x_j
\lesssim
r^{1-\frac{q_j}{p_j}}.
\end{align*}
Thus, 
\begin{align*}
\|f\chi_Q\|_{\vec{q}}
\lesssim
\prod_{j=1}^n
\left(
r^{1-\frac{q_j}{p_j}}
\right)^{\frac{1}{q_j}}
=
\prod_{j=1}^n
r^{{\frac{1}{q_j}}-\frac{1}{p_j}}
=
r^{\sum_{j=1}^n{\frac{1}{q_j}}-\sum_{j=1}^n\frac{1}{p_j}}.
\end{align*}
Since
$\sum_{j=1}^n\frac{1}{p_j}=\frac{n}{p}$,
\[
r^{\frac{n}{p}-\sum_{j=1}^n{\frac{1}{q_j}}}
\|f\chi_Q\|_{\vec{q}}
\lesssim
1.
\]
Taking supremum over all the cubes, we obtain
\[
\|f\|_{ \mathcal{M}^p_{\vec{q}}(\mathbb{R}^n)}
\lesssim
1,
\]
that is, 
\[
f(x)=\prod_{j=1}^n |x_j|^{-\frac{1}{p_j}} \in \mathcal{M}^p_{\vec{q}}(\mathbb{R}^n).
\]
\end{example}

\begin{remark}
In Example \ref{ex 171120-1}, condition (\ref{eq 171120-5})
is a necessary and sufficient condition for 
$f(x)=\prod_{j=1}^n |x_j|^{-\frac{1}{p_j}}$ to be a member in $\mathcal{M}^p_{\vec{q}}(\mathbb{R}^n).$
In fact, let $f \in \mathcal{M}^p_{\vec{q}}(\mathbb{R}^n)$ and $f \neq 0$.
Applying Proposition \ref{prop 171120-2}, we have
\begin{equation} \label{eq 171120-3}
\|f(t\cdot)\|_{\mathcal{M}^p_{\vec{q}}(\mathbb{R}^n)}
=t^{-\frac{n}{p}}\|f\|_{\mathcal{M}^p_{\vec{q}}(\mathbb{R}^n)} \quad (t>0).
\end{equation}
On the other hand, since $f(tx)=t^{-\sum_{j=1}^n\frac{1}{p_j}}f(x)$,
\begin{equation} \label{eq 171120-4}
\|f(t\cdot)\|_{\mathcal{M}^p_{\vec{q}}(\mathbb{R}^n)}
=
t^{-\sum_{j=1}^n\frac{1}{p_j}}
\|f\|_{\mathcal{M}^p_{\vec{q}}(\mathbb{R}^n)}.
\end{equation}
By (\ref{eq 171120-3}) and (\ref{eq 171120-4}), for all $t>0$,
\[
t^{-\sum_{j=1}^n \frac{1}{p_j}}=t^{-\frac{n}{p}}.
\]
Thus, we obtain (\ref{eq 171120-5}).
\end{remark}

\begin{example}
Let $Q$ be a cube and $\vec{q} \in (0,\infty]^n$. Then,
\[
\|\chi_Q\|_{ \mathcal{M}^p_{\vec{q}}(\mathbb{R}^n)}
=|Q|^{\frac{1}{p}}.
\]
To check this, put $\sum_{j=1}^n\frac{1}{q_j}=\bar{q}$. 
First, using Example \ref{ex 171108-1}, we get
\begin{align*}
\|\chi_Q\|_{ \mathcal{M}^p_{\vec{q}}(\mathbb{R}^n)}
=
\sup_{R \in \mathcal{Q}}
|R|^{\frac{1}{p}-\frac{\bar{q}}{n}}
\|\chi_{Q}\chi_{R}\|_{\vec{q}}
\ge
|Q|^{\frac{1}{p}-\frac{\bar{q}}{n}}
\|\chi_{Q}\|_{\vec{q}}
=
|Q|^{\frac{1}{p}-\frac{\bar{q}}{n}}
|Q|^{\frac{\bar{q}}{n}}
=
|Q|^{\frac{1}{p}}.\\
\end{align*}
On the other hand, by Proposition \ref{prop 180521-1},
\[
\|\chi_Q\|_{ \mathcal{M}^p_{\vec{q}}(\mathbb{R}^n)}
\le
\|\chi_Q\|_{ \mathcal{M}^p_{\max(q_1,\ldots,q_n)}(\mathbb{R}^n)} 
=|Q|^\frac1p. 
\]
 
Combining the above two inequalities, we obtain 
\[
\|\chi_Q\|_{ \mathcal{M}^p_{\vec{q}}(\mathbb{R}^n)}
=|Q|^{\frac{1}{p}}.
\]
\end{example}

\section{Proof of Theorems \ref{thm 171206-1} and \ref{thm 180121-1}} 
\label{sec iterated}

In this section,  we investigate the boundedness of the iterated maximal operator in 
$L^{\vec{p}}({\mathbb R}^n)$ and $\mathcal{M}^p_{\vec{q}}(\mathbb{R}^n)$.
First, to show Theorem \ref{thm 171206-1}, we need a lemma due to Bagby in 1975 \cite{Bagby}. 

 
\begin{lemma}{\rm (\cite{Bagby})} \label{lem 180113-4}
Let $1<q_i<\infty (i=1, \ldots,m)$ and $1<p<\infty$.
Let $(\Omega_i, \mu_i) (i=1, \ldots, m)$ be $\sigma$-finite measure spaces, 
and let $\Omega=\Omega_1 \times \cdots \times \Omega_m$.
For $f \in L^0(\mathbb{R}^n\times \Omega)$,
\[
\int_{\mathbb{R}^n}\left\|Mf(x,\cdot)\right\|_{(q_1,\ldots,q_m)}^p{\rm d}x
\lesssim
\int_{\mathbb{R}^n}\left\|f(x,\cdot)\right\|_{(q_1,\ldots,q_m)}^p{\rm d}x,
\]

\end{lemma}
Let us show Theorem \ref{thm 171206-1}.

\begin{proof}
Since
\begin{align*}\label{eq 171103-2}
\|{\mathcal M}_tf\|_{\vec{p}} 
&=\left\|
\left(M_n \cdots M_1 \left[|f|^t\right]\right)^{\frac{1}{t}}
\right\|_{\vec{p}}
=
\left\|
M_n \cdots M_1\left[|f|^t\right]
\right\|_{(\frac{p_1}{t}, \ldots, \frac{p_n}{t})}^{\frac{1}{t}},
\end{align*}
we have only to check (\ref{eq 171028-2}) for $t=1$ and $1<\vec{p}<\infty$.  \nonumber\\
Let $t=1$. Then the conclution can be written as
\[
\|{\mathcal M}_1f\|_{\vec{p}}=\|M_n \cdots M_1 f\|_{\vec{p}} \lesssim \|f\|_{\vec{p}}.
\]
We use induction on $n$.
Let $n=1$. Then, the result follows 
by the classical case of the boundedness of the Hardy--Littlewood maximal operator.\\
Suppose that the result holds for $n-1$, that is, for $h \in L^0({\mathbb R}^{n-1})$ and $1<(q_1, \ldots, q_{n-1})<\infty$,
\[
\|M_{n-1} \cdots M_{1} h\|_{(q_1, \ldots, q_{n-1})} \lesssim \|h\|_{(q_1, \ldots, q_{n-1})}.
\]
By Lemma \ref{lem 180113-4},
\begin{equation}\label{eq 180114-3}
\|M_n f\|_{\vec{p}}
=
\left\|\left[\left\|M_nf\right\|_{(p_1,\ldots,p_{n-1})}\right]\right\|_{(p_n)} 
\lesssim
\left\|\left[\left\|f\right\|_{(p_1,\ldots,p_{n-1})}\right]\right\|_{(p_n)} 
=
\|f\|_{\vec{p}}.
\end{equation}
Thus, by induction assumption, we obtain
\begin{align*}
\|M_nM_{n-1}\cdots M_1 f\|_{\vec{p}}
&=
\|M_n[M_{n-1}\cdots M_1 f]\|_{\vec{p}}\\
&\lesssim
\|M_{n-1}\cdots M_1 f\|_{\vec{p}}\\
&=
\left\|\left\|M_{n-1}\cdots M_1 f\right\|_{(p_1,\ldots,p_{n-1})}\right\|_{p_n}\\
&\lesssim
\left\|\left\|f\right\|_{(p_1,\ldots,p_{n-1})}\right\|_{p_n}
\lesssim
\|f\|_{\vec{p}}.
\end{align*}

\end{proof}

\begin{remark}
In 1935, Jessen, Marcinkiewicz and Zygmund showed the boundedness of the iterated maximal operator 
in the classical $L^p$ spaces \cite{J-M-Z}.
Also, Bagby showed the boundedness of the Hardy--Littlewood maximal operator for the functions taking values
in mixed Lebesgue spaces in 1975 \cite{Bagby}.
St\"ockert showed the boundedness of the iterated maximal operator $\mathcal{M}_1$ in 1978 \cite{St}.
But, the proof is not correct.
Its proof uses the following estimate:
\begin{equation}\label{eq 180114-1}
\|M_kf\|_{(q_j)}\le M_k\|f\|_{(q_j)}
\end{equation}
for $f \in L^0(\mathbb{R}^n)$ and $1<\vec{q}<\infty$.
We disprove this estimate by an example:
For the sake of simplicity, let $n=2, k=2$ and $j=1$. That is, (\ref{eq 180114-1}) implies
\[
\|M_2f\|_{(q_1)}\le M_2\|f\|_{(q_1)}.
\]
Let $0<q_2<1< q_1$ and $q_1q_2>1$. 
Define the function $\varphi$ as follows:
\[
\varphi(t)=t^{-\frac{q_1}{q_2}}\chi_{(0, 1)}(t) \quad (t \in \mathbb{R}).
\]
Let
\[
f(x,y)=\chi_{E}(x,y),
\]
where $E=\{(x, y): 0\le x \le \varphi(y)\}$.
Let $0 < y \le 1$.
First, we calculate $M_2\|f\|_{(q_1)}$.
Since
\[
\|f\|_{(q_1)}
=\left(\int_{\mathbb{R}}\chi_E(x,y) {\rm d} x \right)^{\frac{1}{q_1}}
=\varphi(y)^{\frac{1}{q_1}}
=y^{-\frac{1}{q_1q_2}}\chi_{(0,1)}(y),
\]
we get
\[
M_2\|f\|_{(q_1)}(y)
=\frac{1}{y}\int_0^y t^{-\frac{1}{q_1q_2}} {\rm d}t
=\frac1y\frac{q_1q_2}{q_1q_2-1}y^{1-\frac{1}{q_1q_2}}
= \frac{q_1q_2}{q_1q_2-1}y^{-\frac{1}{q_1q_2}}.
\]
Next, we calculate $M_2f(x,y)$.
Since
\[
E=\{(x, y): 0\le x \le \varphi(y)\}=\{(x, y): 0\le x,  0\le y\le \min(1, x^{-\frac{q_2}{q_1}})\},
\]
we have
\begin{align*}
M_2f(x,y)=\frac1y\int_0^y\chi_E(x,t) {\rm d}t
=\frac1y\int_0^{\min(y, x^{-\frac{q_2}{q_1}})}\chi_{(x\ge0)}(x) {\rm d}t
=\frac{\min(y, x^{-\frac{q_2}{q_1}})\chi_{(x\ge0)}(x)}{y}.
\end{align*}
Thus,
\begin{align*}
\|M_2f(\cdot,y)\|_{(q_1)}^{q_1}
&=\frac1y\int_{\mathbb{R}} \left[\min(y, x^{-\frac{q_2}{q_1}})\right]^{q_1} {\rm d}x\\
&= \frac1y\int_{\mathbb{R}}\left[ y\chi_{\left(0\le x\le y^{-\frac{q_1}{q_2}}\right)}(x)+
x^{-\frac{q_2}{q_1}} \chi_{\left(x\ge y^{-\frac{q_1}{q_2}}\right)}(x)\right]^{q_1}{\rm d}x\\
&\ge\frac1y\int_{y^{-\frac{q_1}{q_2}}}^\infty x^{-q_2} {\rm d}x=\infty.
\end{align*}
Therefore, (\ref{eq 180114-1}) does not hold.
On the other hand, using Lemma \ref{lem 180113-4}, we give a correct proof for Theorem \ref{thm 171206-1}.
\end{remark}

Moreover, we shall consider why we investigate the iterated maximal operator.
\begin{example}
Let $\mathcal{R}$ be a set of all rectangles in $\mathbb{R}^n$. 
By $M_R$, denote the strong maximal operator which is generated by a rectangle $R$:
for $f \in L^0(\mathbb{R}^n)$,
\[
M_Rf(x)
=\sup_{R \in \mathcal{R}}
\frac{\chi_R(x)}{|R|}
\int_{R}
|f(y)|
{\rm d}y.
\]
Then, the followings follow \cite{J-M-Z}:
\[
M_Rf(x)\le M_n \cdots M_1f (x)={\mathcal M}_1f(x),
\]
and
\[
M_Rf(x)\le M_1 \cdots M_nf (x),
\]
and so on.
Thus, the iterated maximal operator can controll the strong maximal operator.
On the other hand, the relation between $M_1\cdots M_n$ and $M_n \cdots M_1$ is not comparable. 
To see this, we give the following example.
For the sake of simplicity, let $n=2$.
Let $f(x,y)=\chi_{\Delta}(x,y)$, where 
\[
\Delta=\{(x,y): 0\le x\le1, 0\le y\le x\}.  
\]
First, we calculate $M_1f$ and $M_2f$:
\begin{eqnarray*}
M_1f(x,y)=
\left\{\begin{array}{ll}
0 & (y\le0, 1\le y),\\
1 & (0\le y\le1, y\le x),\\
\frac{1-y}{1-x} & (0 \le y \le1, x\le y),\\
\frac{1-y}{x-y} & (0\le y \le1, 1\le x),\\
\end{array} \right.
\end{eqnarray*}
and
\begin{eqnarray*}
M_2f(x,y)=
\left\{\begin{array}{ll}
0 & (x\le0, 1\le x),\\
\frac{x}{x-y} & (0\le x\le1, y\le0),\\
1 & (0 \le x \le1, 0\le y\le x),\\
\frac{x}{y} & (0\le x \le1, x\le y).\\
\end{array} \right.
\end{eqnarray*}
Next, we calculate $M_2M_1f$ and $M_1M_2f$.
In particular, we consider two cases. For $0\le x \le 1, y\ge1$, we get 
\[
M_2M_1f(x,y)=\frac{-x^2-y^2+2y}{2y(1-x)}, \quad M_1M_2f(x,y)=\frac{x+1}{2y}.
\]
For $x \ge 1, 0\le y\le1$, we have
\[
M_2M_1f(x,y)=\frac{1}{y}\left(y+(x-1)\log\frac{x-1}{x}\right), \quad M_1M_2f(x,y)=\frac{x-\sqrt{x^2-1}}{y}.
\]
Thus, we obtain
\[
M_2M_1f \le M_1M_2f \quad (0\le x \le 1, y\ge1),
\]
while
\[
M_2M_1f \ge M_1M_2f \quad (x\ge1, 0\le y\le1).
\]

\end{example}

Next, we consider the boundedness of the maximal operator in classical and mixed Morrey spaces.
The following proposition is important when we show the boundedness of the Hardy--Littlewood maximal operator
in classical and mixed Morrey spaces.

\begin{proposition}\label{ex 171103-1} {\rm (\cite{SHG15} Lemma 4.2)}
For all measurable functions $f$ and cubes $Q$,
we have
\begin{equation}\label{eq:131109-69}
M[\chi_{{\mathbb R}^n \setminus 5Q}f](y)
\lesssim \sup_{Q \subset R \in {\mathcal Q}}
\frac{1}{|R|}\int_R |f(x)| {\rm d}x
\quad (y \in Q).
\end{equation}
\end{proposition}

First, we prove the boundedness of the Hardy--Littlewood maximal operator in mixed Morrey spaces.
The boundedness of the Hardy--Littlewood maximal operator in classical Morrey spaces is showed by 
Chiarenza and Frasca in 1987 \cite{C-F}.

\begin{theorem}\label{thm 180122-1}
Let $1<\vec{q}<\infty$ and $1<p\le\infty$ satisfy $\frac np \le\sum_{j=1}^n\frac1q_j$. Then
\[
\|Mf\|_{\mathcal{M}^p_{\vec{q}}(\mathbb{R}^n)} \lesssim \|f\|_{\mathcal{M}^p_{\vec{q}}(\mathbb{R}^n)}
\]
for all $f \in L^0(\mathbb{R}^n)$.
\end{theorem}

\begin{proof}
It suffices to verify that, for any cube $Q=Q(x,r)$,
\[
|Q|^{\frac{1}{p}-\frac{1}{n}
\left(
\sum_{j=1}^n\frac{1}{q_j}
\right)}
\|(Mf)\chi_Q\|_{\vec{q}}
\lesssim \|f\|_{\mathcal{M}^p_{\vec{q}}(\mathbb{R}^n)}.
\]
Now, we decompose
\[
|f(y)|
=
\chi_{Q(x, 5r)}(y)|f(y)|+\chi_{Q(x, 5r)^c}(y)|f(y)|
\equiv
f_1(y)+f_2(y) \quad(y \in \mathbb{R}^n).
\]
Using the subadditivity of $M$, we obtain
\[
Mf(y) \le Mf_1(y)+Mf_2(y) \quad(y \in \mathbb{R}^n).
\]
First, the boundedness of $M$ on the mixed Lebesgue space $L^{\vec{q}}(\mathbb{R}^n)$ \cite{Bagby}
yields
\begin{align*}
|Q|^{\frac{1}{p}-\frac{1}{n}
\left(
\sum_{j=1}^n\frac{1}{q_j}
\right)}
\|(Mf_1)\chi_Q\|_{\vec{q}}
&\le
|Q|^{\frac{1}{p}-\frac{1}{n}
\left(
\sum_{j=1}^n\frac{1}{q_j}
\right)}
\|Mf_1\|_{\vec{q}}\\
&\lesssim
|Q|^{\frac{1}{p}-\frac{1}{n}
\left(
\sum_{j=1}^n\frac{1}{q_j}
\right)}
\|f_1\|_{\vec{q}}\\
&=
|Q|^{\frac{1}{p}-\frac{1}{n}
\left(
\sum_{j=1}^n\frac{1}{q_j}
\right)}
\|f\chi_{Q(x, 5r)}\|_{\vec{q}}\\
&=
|Q(x, 5r)|^{\frac{1}{p}-\frac{1}{n}
\left(
\sum_{j=1}^n\frac{1}{q_j}
\right)}
\|f\chi_{Q(x, 5r)}\|_{\vec{q}}\\
&\le
\|f\|_{\mathcal{M}^p_{\vec{q}}(\mathbb{R}^n)}.
\end{align*}
Second, by Proposition \ref{ex 171103-1}, we get
\[
Mf_2(y)
=M[\chi_{{\mathbb R}^n \setminus 5Q}f](y)
\lesssim \sup_{Q \subset R \in {\mathcal Q}}
\frac{1}{|R|}\int_R |f(x)| {\rm d}x
\quad (y \in Q).
\]
Thus, we see that
\begin{align} \label{eq 180517-1}
&|Q|^{\frac{1}{p}-
\frac{1}{n}\left(\sum_{j=1}^n\frac{1}{q_j}\right)}
\|(M f_2)\chi_Q\|_{\vec{q}} \nonumber\\
&\lesssim
\sup_{Q \subset R \in {\mathcal Q}}
|Q|^{\frac{1}{p}-
\frac{1}{n}\left(\sum_{j=1}^n\frac{1}{q_j}\right)}
\left\|
\frac{1}{|R|}\int_R |f(x)| {\rm d}x
\times\chi_Q\right\|_{\vec{q}}.
\end{align}
Thanks to Example \ref{ex 171108-1}, we have
\begin{align*}
(\ref{eq 180517-1})&=
\sup_{Q\subset R \in {\mathcal Q}}
|Q|^{\frac{1}{p}-
\frac{1}{n}\left(\sum_{j=1}^n\frac{1}{q_j}\right)}
\frac{1}{|R|}\int_R |f(x)| {\rm d}x
\times\|\chi_Q\|_{\vec{q}}\\
&=
\sup_{Q\subset R \in {\mathcal Q}}
|Q|^{\frac{1}{p}-
\frac{1}{n}\left(\sum_{j=1}^n\frac{1}{q_j}\right)}
\frac{1}{|R|}\int_R |f(x)| {\rm d}x
\times
|Q|^{\frac{1}{n}\left(\sum_{j=1}^n\frac{1}{q_j}\right)}\\
&\le
\sup_{R \in {\mathcal Q}}|R|^{\frac{1}{p}-1}\int_R|f(x)|{\rm d}x.
\end{align*}
By Proposition \ref{prop 180521-1}, taking into account
${\mathcal M}^p_{\vec{q}}({\mathbb R}^n) 
\hookrightarrow {\mathcal M}^p_{(\underbrace{1, \ldots, 1}_{\mbox{$n$ times}})}({\mathbb R}^n)
= {\mathcal M}^p_1({\mathbb R}^n)$
with embedding constant $1$,
we get
\begin{align*}
|Q|^{\frac{1}{p}-
\frac{1}{n}\left(\sum_{j=1}^n\frac{1}{q_j}\right)}
\|(M f_2)\chi_Q\|_{\vec{q}}
\le
\|f \|_{{\mathcal M}^p_1(\mathbb{R}^n)}
\le
\|f \|_{{\mathcal M}^p_{\vec{q}}(\mathbb{R}^n)}.
\end{align*}
Thus, taking the supremum over all the cubes, we obtain
\[
\|Mf_2 \|_{{\mathcal M}^p_{\vec{q}}(\mathbb{R}^n)}
\lesssim
\|f \|_{{\mathcal M}^p_{\vec{q}}(\mathbb{R}^n)}.
\]
Hence, the result holds.
\end{proof}

Next, we show the boundedness of the iterated maximal operator for mixed Morrey spaces.
To show this, we need auxiliary estimates.

\begin{lemma} \label{lem 180115-1}
Let $\{f_{(j_1, \ldots, j_m)}\}_{j_1, \ldots, j_m=0}^\infty \subset L^0(\mathbb{R}^n)$ and $w \in A_p$.
Then, for $1<q_i \le \infty (i=1, \ldots,m)$ and $1<p<\infty$, 
\begin{equation*}
\left\|\left[\left\|Mf_{(j_1,\ldots,j_m)}\right\|_{\ell^{(q_1,\ldots,q_m)}}w^\frac{1}{p}\right]\right\|_p
\lesssim
\left\|\left[\left\|f_{(j_1,\ldots,j_m)}\right\|_{\ell^{(q_1,\ldots,q_m)}}w^\frac{1}{p}\right]\right\|_p,
\end{equation*}
that is, 
\begin{eqnarray*}
&&
\left\|
\left(
\sum_{j_m=1}^\infty
\left(\cdots
\sum_{j_2=1}^\infty
\left(
\sum_{j_1=1}^\infty
(M f_{(j_1,\ldots,j_m)})^{q_1}
\right)^{\frac{q_2}{q_1}}
\cdots
\right)^{\frac{q_m}{q_{m-1}}}
\right)^{\frac{1}{q_m}}w^\frac{1}{p}\right\|_p\\
&&\lesssim
\left\|\left(
\sum_{j_m=1}^\infty
\left(\cdots
\sum_{j_2=1}^\infty
\left(
\sum_{j_1=1}^\infty
|f_{(j_1,\ldots,j_m)}|^{q_1}
\right)^{\frac{q_2}{q_1}}
\cdots
\right)^{\frac{q_m}{q_{m-1}}}
\right)^{\frac{1}{q_m}}w^\frac{1}{p}\right\|_p.
\end{eqnarray*}
 
\end{lemma}

\begin{proof}
We induct on $m$.
Let $m=1$. Then, this is the weighted Fefferman--Stein maximal inequality \cite{A-J}.
Suppose that the result holds for $m-1$:
\begin{eqnarray*}
\left\|\left[\left\|Mf_{(j_1,\ldots,j_{m-1})}\right\|_{\ell^{(q_1,\ldots,q_{m-1})}}w^\frac{1}{p}\right]\right\|_p
\lesssim
\left\|\left[\left\|f_{(j_1,\ldots,j_{m-1})}\right\|_{\ell^{(q_1,\ldots,q_{m-1})}}w^\frac{1}{p}\right]\right\|_p.
\end{eqnarray*}
Let $p=q_m$. We calculate
\begin{align*}
\left\|\left[\left\|Mf_{(j_1,\ldots,j_m)}\right\|_{\ell^{(q_1,\ldots,q_m)}}w^\frac{1}{p}\right]\right\|_p^p
&=
\left\|\left[\left\|Mf_{(j_1,\ldots,j_m)}\right\|_{\ell^{(q_1,\ldots,q_{m-1},1)}}^pw\right]\right\|_1.
\end{align*}
By the Lebesgue convergence theorem, 
\begin{align*}
\left\|\left[\left\|Mf_{(j_1,\ldots,j_m)}\right\|_{\ell^{(q_1,\ldots,q_m)}}w^\frac{1}{p}\right]\right\|_p^p
&=
\sum_{j_m=1}^\infty
\left\|\left[\left\|Mf_{(j_1,\ldots,j_m)}\right\|_{\ell^{(q_1,\ldots,q_{m-1})}}^pw\right]\right\|_1\\
&=
\sum_{j_m=1}^\infty
\left\|\left[\left\|Mf_{(j_1,\ldots,j_m)}\right\|_{\ell^{(q_1,\ldots,q_{m-1})}}w^\frac{1}{p}\right]\right\|_p^p.
\end{align*}
Using induction assumption, we obtain
\begin{align*}
\left\|\left[\left\|Mf_{(j_1,\ldots,j_m)}\right\|_{\ell^{(q_1,\ldots,q_m)}}w^\frac{1}{p}\right]\right\|_p^p
&\lesssim
\sum_{j_m=1}^\infty
\left\|\left[\left\|f_{(j_1,\ldots,j_m)}\right\|_{\ell^{(q_1,\ldots,q_{m-1})}}w^\frac{1}{p}\right]\right\|_p^p\\
&=
\sum_{j_m=1}^\infty
\left\|\left[\left\|f_{(j_1,\ldots,j_m)}\right\|_{\ell^{(q_1,\ldots,q_{m-1})}}^pw\right]\right\|_1\\
&=
\left\|\left[\left\|f_{(j_1,\ldots,j_m)}\right\|_{\ell^{(q_1,\ldots,q_{m-1},1)}}^pw\right]\right\|_1\\
&=
\left\|\left[\left\|f_{(j_1,\ldots,j_m)}\right\|_{\ell^{(q_1,\ldots,q_m)}}w^\frac{1}{p}\right]\right\|_p^p.
\end{align*}
Thus, the result holds when $p=q_m$.
Using Proposition \ref{thm:161125-1}, we conclude the result for all $1<p<\infty$.
\end{proof}

\begin{lemma} \label{lem 180119-2}
Let $\{f_{(j_1, \ldots, j_m)}\}_{j_1, \ldots, j_m=1}^\infty \subset L^0(\mathbb{R}^n)$ and $w_k \in A_{q_k}(\mathbb{R})$.
Then, for $1<q_i\le \infty (i=1, \ldots,m)$ and $k=1, \ldots,n$, 
\begin{align}\label{eq 180119-3}
\lefteqn{\nonumber
\left\|\left[
\left\|\left[\left\|M_kf_{(j_1,\ldots,j_m)}\right\|_{\ell^{(q_1,\ldots,q_k)}}\right]w_k(\cdot_k)^\frac{1}{q_k}\right\|_{q_k}
\right]\right\|_{\ell^{(q_{k+1},\ldots,q_m)}}}\\ 
&\lesssim
\left\|\left[
\left\|\left[\left\|f_{(j_1,\ldots,j_m)}\right\|_{\ell^{(q_1,\ldots,q_k)}}\right]w_k(\cdot_k)^\frac{1}{q_k}\right\|_{q_k}
\right]\right\|_{\ell^{(q_{k+1},\ldots,q_m)}}.
\end{align}
\end{lemma}

\begin{proof}
By Lemma \ref{lem 180115-1}, 
\[
\left\|\left[\left\|M_kf_{(j_1,\ldots,j_m)}\right\|_{\ell^{(q_1,\ldots,q_k)}}\right]w_k(\cdot_k)^\frac{1}{q_k}\right\|_{q_k}
\lesssim
\left\|\left[\left\|f_{(j_1,\ldots,j_m)}\right\|_{\ell^{(q_1,\ldots,q_k)}}\right]w_k(\cdot_k)^\frac{1}{q_k}\right\|_{q_k}.
\]
Taking $\ell^{(q_{k+1},\ldots,q_m)}$-norm for $j_{k+1},\ldots,j_m$, we conclude (\ref{eq 180119-3}).
\end{proof}

\begin{lemma}\label{lem 180119-4}
Let $f \in L^0(\mathbb{R}^n)$ and $w_n \in A_{p_n}(\mathbb{R})$.
Then, for $1<p_i\le \infty (i=1, \ldots,n)$, 
\begin{equation}
\left\|\left[
\left\|M_nf\right\|_{(p_1,\ldots,p_{n-1})}\right]
w_n(\cdot_n)^\frac{1}{p_n}\right\|_{(p_n)}
\lesssim
\left\|\left[
\left\|f\right\|_{(p_1,\ldots,p_{n-1})}\right]
w_n(\cdot_n)^\frac{1}{p_n}\right\|_{(p_n)}.
\end{equation}

\end{lemma}

\begin{proof}
Assume that
$f \in L^{\vec{p}}({\mathbb R}^n)$ is a function of the form:
\[
f(x_1,\ldots,x_n)
=
\sum_{m' \in {\mathbb Z}^{n-1}}
\chi_{m'+[0,1]^{n-1}}(r x')f_{m'}(x_n),
\]
where $r>0$ and $\{f_{m'}\}_{m' \in {\mathbb Z}^{n-1}} \subset L^0({\mathbb R})$.
Then
\[
M_n f(x_1,\ldots,x_n)
=
\sum_{m' \in {\mathbb Z}^{n-1}}
\chi_{m'+[0,1]^{n-1}}(r x')M_n f_{m'}(x_n),
\]
since the summand is made up of at most one non-zero function
once we fix $x'$.
Define $v>0$ by
\[
\frac{1}{v}=\frac{1}{p_1}+\cdots+\frac{1}{p_{n-1}}.
\]
Then, by Proposition \ref{prop 180113-2},
\begin{align*}
&\left\|
\left\|
M_nf
\right\|_{(p_1, \ldots, p_{n-1})}
w_n(\cdot_n)^\frac{1}{p_n}\right\|_{(p_n)}\\
&=
\left\|\left\|
\sum_{m' \in {\mathbb Z}^{n-1}}
\chi_{m'+[0,1]^{n-1}}(r x')M_n f_{m'}(\cdot_n)
\right\|_{(p_1, \ldots, p_{n-1})}
w_n(\cdot_n)^\frac{1}{p_n}\right\|_{(p_n)}\\
&=
r^{-\frac1v}
\left\|\left\|
\sum_{m' \in {\mathbb Z}^{n-1}}
\chi_{m'+[0,1]^{n-1}}(x')M_n f_{m'}(\cdot_n)
\right\|_{(p_1, \ldots, p_{n-1})}
w_n(\cdot_n)^\frac{1}{p_n}\right\|_{(p_n)}.
\end{align*}
Thus, by Lemma \ref{lem 180115-1},
\begin{align*}
\left\|
\left\|
M_nf
\right\|_{(p_1, \ldots, p_{n-1})}
w_n(\cdot_n)^\frac{1}{p_n}\right\|_{(p_n)}
&=
r^{-\frac1v}
\left\|\left\|
M_n f_{m'}(\cdot_n)
\right\|_{\ell^{(p_1, \ldots, p_{n-1})}}
w_n(\cdot_n)^\frac{1}{p_n}\right\|_{(p_n)}\\
&\lesssim
r^{-\frac1v}
\left\|\left\|
f_{m'}(\cdot_n)
\right\|_{\ell^{(p_1, \ldots, p_{n-1})}}
w_n(\cdot_n)^\frac{1}{p_n}\right\|_{(p_n)}\\
&=
\left\|
\left\|
f
\right\|_{(p_1, \ldots, p_{n-1})}
w_n(\cdot_n)^\frac{1}{p_n}\right\|_{(p_n)}.\\
\end{align*}
Let $f \in L^{\vec{p}}({\mathbb R}^n)$ be arbitrary.
Write
\[
f_r(x)=
\frac{1}{r^{n-1}}
\sum_{m' \in {\mathbb Z}^{n-1}}
\chi_{r m'+[0,r]^{n-1}}(x')
\int_{r m'+[0,r]^{n-1}}f(y',x_n)\,dy'.
\]
Thanks to the Lebesgue differentiation theorem
\[
f(x', x_n)=\lim_{r \downarrow 0}f_r(x', x_n)
\]
for almost every $x' \in {\mathbb R}^{n-1}$.
Thus, by the Fatou lemma, we obtain
\[
M_n f(x) \le \liminf_{r \downarrow 0}M_n f_r(x).
\]
Meanwhile, for all $r>0$, since $f_r\le M_{n-1}\cdots M_{1}f$,
by Theorem \ref{thm 171206-1}, we get
\begin{align*}
\left\|\left\|f_r\right\|_{\vec{s}}w_n(\cdot_n)^\frac{1}{p_n}\right\|_{(p_n)}
\le
\left\|\left\| M_{n-1}\cdots M_{1}f\right\|_{\vec{s}}w_n(\cdot_n)^\frac{1}{p_n}\right\|_{(p_n)}
\lesssim
\left\|\left\|f\right\|_{\vec{s}}w_n(\cdot_n)^\frac{1}{p_n}\right\|_{(p_n)},
\end{align*}
where $\vec{s}=(p_1,\ldots, p_{n-1})$.
As a consequence,
by the Lebesgue differentiation theorem and the Fatou lemma,
we obtain
\begin{align*}
\left\|
\left\|
M_nf
\right\|_{\vec{s}}
w_n(\cdot_n)^\frac{1}{p_n}\right\|_{(p_n)}
&\le
\liminf_{r \downarrow 0}
\left\|
\left\|
M_nf_r
\right\|_{\vec{s}}
w_n(\cdot_n)^\frac{1}{p_n}\right\|_{(p_n)}\\
&\lesssim
\liminf_{r \downarrow 0}
\left\|
\left\|
f_r
\right\|_{\vec{s}}
w_n(\cdot_n)^\frac{1}{p_n}\right\|_{(p_n)}
\lesssim
\left\|\left\|f\right\|_{\vec{s}}w_n(\cdot_n)^\frac{1}{p_n}\right\|_{(p_n)}.
\end{align*}

\end{proof}

\begin{proposition} \label{prop 180119-6}
Let $1<\vec{q}<\infty$. Let $f \in L^0(\mathbb{R}^n)$ and  $w_k \in A_{q_k}(\mathbb{R})$ for $k=1,\ldots,n$.
Then, 
\[
\left\|
\mathcal{M}_1f\cdot\bigotimes_{k=1}^nw_{k}^{\frac{1}{q_k}}
\right\|_{\vec{q}}
\lesssim
\left\|
f\cdot\bigotimes_{k=1}^nw_{k}^{\frac{1}{q_k}}
\right\|_{\vec{q}}.
\]
\end{proposition}

\begin{proof}
We induct on $n$.
Let $n=1$. Then, the result is the boundedness of the Hardy--Littlewood maximal operator on weighted $L^p$ spaces. 
Suppose that the result holds for $n-1$:
\[
\left\|
(M_{n-1}\cdots M_1h)\cdot\bigotimes_{k=1}^{n-1}w_{k}^{\frac{1}{q_k}}
\right\|_{(q_1,\ldots,q_{n-1})}
\lesssim
\left\|
h\cdot\bigotimes_{k=1}^{n-1}w_{k}^{\frac{1}{q_k}}
\right\|_{(q_1,\ldots,q_{n-1})}.
\]
By Lemma \ref{lem 180119-4}, we obtain
\begin{align*}
\left\|
\mathcal{M}_1f\cdot\bigotimes_{k=1}^nw_{k}^{\frac{1}{q_k}}
\right\|_{\vec{q}}
&=
\left\|
(M_n\cdots M_1f)\cdot\bigotimes_{k=1}^nw_{k}^{\frac{1}{q_k}}
\right\|_{\vec{q}}\\
&=
\left\|\left[
\left\|
M_n\left([M_{n-1}\cdots M_1f]\cdot\bigotimes_{k=1}^{n-1}w_{k}^{\frac{1}{q_k}}\right)
\right\|_{(q_1,\ldots,q_{n-1})}\right]
w_n(\cdot_n)^{\frac{1}{q_n}}\right\|_{(q_n)}\\
&\lesssim
\left\|\left[
\left\|
[M_{n-1}\cdots M_1f]\cdot\bigotimes_{k=1}^{n-1}w_{k}^{\frac{1}{q_k}}
\right\|_{(q_1,\ldots,q_{n-1})}\right]
w_n(\cdot_n)^{\frac{1}{q_n}}\right\|_{(q_n)}\\
&\lesssim
\left\|\left[
\left\|
f\cdot\bigotimes_{k=1}^{n-1}w_{k}^{\frac{1}{q_k}}
\right\|_{(q_1,\ldots,q_{n-1})}\right]
w_n(\cdot_n)^{\frac{1}{q_n}}\right\|_{(q_n)}
=
\left\|
f\cdot\bigotimes_{k=1}^nw_{k}^{\frac{1}{q_k}}
\right\|_{\vec{q}}.
\end{align*}
\end{proof}

\begin{proposition} \label{prop 180119-5}
Let $0<p<\infty$, $0<\vec{q} \le \infty$ and $\eta \in {\mathbb R}$ satisfy
\[
0<\sum_{j=1}^n \frac{1}{q_j}-\frac{n}{p}<\eta<1.
\]
Then, for $f \in L^0(\mathbb{R}^n)$
\[
\|f\|_{{\mathcal M}^p_{\vec{q}}}
\sim
\sup_{Q \in {\mathcal Q}}
|Q|^{\frac1p-\frac1n\sum_{j=1}^n \frac{1}{q_j}}
\|f(\mathcal{M}_1\chi_Q)^\eta\|_{\vec{q}}.
\]
\end{proposition}

\begin{proof}
One inequality is clear:
\[
\|f\|_{{\mathcal M}^p_{\vec{q}}}
\le
\sup_{Q \in {\mathcal Q}}
|Q|^{\frac1p-\frac1n\sum_{j=1}^n \frac{1}{q_j}}
\|f(\mathcal{M}_1\chi_Q)^\eta\|_{\vec{q}}.
\]
We need to show the opposite inequality.
To this end, we fix a cube $Q=I_1 \times \cdots \times I_n$.
Given $(l_1,\ldots,l_n)\in {\mathbb N}^n$,
we write $l=\max(l_1,\ldots,l_n)$.
Then we have
\begin{align*}
\lefteqn{
|Q|^{\frac1p-\frac1n\sum_{j=1}^n \frac{1}{q_j}}
\|f(\mathcal{M}_1\chi_Q)^\eta\|_{\vec{q}}
}\\
&\lesssim
|Q|^{\frac1p-\frac1n\sum_{j=1}^n \frac{1}{q_j}}
\left\|f\prod_{j=1}^n\left( \frac{\ell(I_j)}{\ell(I_j)+|\cdot_j-c(I_j)|}\right)^\eta\right\|_{\vec{q}}\\
&\lesssim
|Q|^{\frac1p-\frac1n\sum_{j=1}^n \frac{1}{q_j}}
\sum_{l_1,\ldots,l_n=1}^\infty
\frac{1}{2^{(l_1+\cdots+l_n)\eta}}
\left\|f\chi_{2^{l_1}I_1 \times \cdots \times 2^{l_n}I_n}\right\|_{\vec{q}}\\
&\lesssim
|Q|^{\frac1p-\frac1n\sum_{j=1}^n \frac{1}{q_j}}
\sum_{l_1,\ldots,l_n=1}^\infty
\frac{1}{2^{(l_1+\cdots+l_n)\eta}}
\left\|f\chi_{2^{l}Q}\right\|_{\vec{q}}\\
&\lesssim
\sum_{l_1,\ldots,l_n=1}^\infty
\frac{2^{\frac{l}{n}\sum_{j=1}^n \frac{1}{q_j}-\frac{2^{l}}{p}}}{2^{(l_1+\cdots+l_n)\eta}}
|2^{l}Q|^{\frac1p-\frac1n\sum_{j=1}^n \frac{1}{q_j}}
\left\|f\chi_{2^{l}Q}\right\|_{\vec{q}}\\
&\lesssim
\|f\|_{{\mathcal M}^p_{\vec{q}}},
\end{align*}
where $c(I_j)$ denotes the center of $I_j$. 
Hence, we obtain the result.
\end{proof}

We recall Theorem \ref{thm 180121-1}.

\begin{theorem} \label{thm 180119-1}
Let $0<\vec{q}\le\infty$ and $0<p<\infty$ satisfy 
\[
\frac np \le\sum_{j=1}^n\frac1q_j,
\quad
\frac{n-1}{n}p<\max(q_1,\ldots,q_n).
\]
 If $0<t<\min(q_1, \ldots, q_n, p)$, then
\begin{equation*}
\|{\mathcal M}_tf\|_{\mathcal{M}^p_{\vec{q}}(\mathbb{R}^n)} 
\lesssim \|f\|_{\mathcal{M}^p_{\vec{q}}(\mathbb{R}^n)}
\end{equation*}
for all $f \in \mathcal{M}^p_{\vec{q}}(\mathbb{R}^n)$.
\end{theorem}

\begin{proof}
We have only to check for $t=1, 1<p<\infty$ and $1<\vec{q} \le \infty$
as we did in Theorem \ref{thm 171206-1}.
For $\eta \in \mathbb{R}$ satisfying
\begin{equation} \label{eq 180119-7}
0<\sum_{j=1}^n \frac{1}{q_j}-\frac{n}{p}<\eta<\frac{1}{\max(q_1,\ldots,q_n)},
\end{equation}
once we show  
\begin{equation}\label{eq 180119-8}
\|\mathcal{M}_1f(\mathcal{M}_1\chi_Q)^\eta\|_{\vec{q}}
\lesssim
\|f(\mathcal{M}_1\chi_Q)^\eta\|_{\vec{q}},
\end{equation}
we get
\[
|Q|^{\frac1p-\frac1n\sum_{j=1}^n \frac{1}{q_j}}
\|\mathcal{M}_1f(\mathcal{M}_1\chi_Q)^\eta\|_{\vec{q}}
\lesssim
|Q|^{\frac1p-\frac1n\sum_{j=1}^n \frac{1}{q_j}}
\|f(\mathcal{M}_1\chi_Q)^\eta\|_{\vec{q}}.
\]
Remark that such an $\eta$ exists because
\[
\frac{n-1}{n}p<\max(q_1,\ldots,q_n).
\]
Taking the supremum for all cubes and using Proposition \ref{prop 180119-5}, we conclude the result.

We shall show (\ref{eq 180119-8}).
Let $Q=I_1 \times I_2 \times \cdots \times I_n$. Then, 
\[
(\mathcal{M}_1\chi_Q)^\eta
=
\left(\bigotimes_{j=1}^nM_j\chi_{I_j}\right)^\eta
=
\bigotimes_{j=1}^n\left(M_j\chi_{I_j}\right)^\eta.
\]
Here, $(M_j\chi_{I_j})^{\eta q_j}$ is an $A_1$-weight if and only if 
$\eta>0$ satisfies
\begin{equation}\label{eq 180119-9}
0\le\eta q_j<1,
\end{equation} 
and so $(M_j\chi_{I_j})^{\eta q_j}\in A_1 \subset A_{q_j}$ for all $q_j$.
Thus, by Proposition \ref{prop 180119-6},
\begin{align*}
\left\|\mathcal{M}_1f(\mathcal{M}_1\chi_Q)^\eta\right\|_{\vec{q}}
&=
\left\|(\mathcal{M}_1f)\bigotimes_{j=1}^n\left(M_j\chi_{I_j}\right)^\eta\right\|_{\vec{q}}
\lesssim
\left\|f\bigotimes_{j=1}^n\left(M_j\chi_{I_j}\right)^\eta\right\|_{\vec{q}}
=
\|f(\mathcal{M}_1\chi_Q)^\eta\|_{\vec{q}}.
\end{align*}
Thus, (\ref{eq 180119-8}) holds.
\end{proof}

In Theorem {\ref{thm 180119-1}}, letting $q_j=q$ for all $j=1, \ldots, n$, we get the following result:

\begin{corollary}
Let 
\[
0<\frac{n-1}{n}p<q\le p<\infty.
\]
 If $0<t<q$, then
\[
\|{\mathcal M}_tf\|_{\mathcal{M}^p_q(\mathbb{R}^n)} 
\lesssim \|f\|_{\mathcal{M}^p_q(\mathbb{R}^n)}
\]
for all $f \in \mathcal{M}^p_q(\mathbb{R}^n)$.
\end{corollary}

\section{Proof of Theorem \ref{thm 171115-1} and related inequalities}
\label{sec dual inequality}

Next, we shall show the dual inequality of Stein type \cite{F-S}
for the iterated maximal operator and $L^{\vec{p}}(\mathbb{R}^n)$.

\begin{proposition}\label{prop 171125-1}
Let $f,w$
be measurable functions.
Suppose in addition that $w \ge 0$
almost everywhere.
Let $1\le i_1, i_2, \cdots, i_k\le n \,(1\le k\le n)$ and $i_j \neq i_k (j \neq k)$.
Then for all
$1<p<\infty$,
\begin{eqnarray*}
\int_{{\mathbb R}^n} M_{i_k}\cdots M_{i_1}f(x)^p\cdot w(x){\rm d}x
\lesssim
\int_{{\mathbb R}^n} |f(x)|^p\cdot  M_{i_1}\cdots M_{i_k}w(x){\rm d}x.
\end{eqnarray*}
\end{proposition}

\begin{proof}
We use induction on $k$.
Let $k=1$. Fix $(x_1, \ldots, x_{i_1-1}, x_{i_1+1}, \ldots, x_n)$. Then, by the dual inequality of Stein type for $M_{i_1}$,
 we get
\begin{equation*}
\int_{{\mathbb R}} M_{i_1}f(x)^p\cdot w(x){\rm d}x_{i_1}
\lesssim
\int_{{\mathbb R}} |f(x)|^p\cdot M_{i_1}w(x){\rm d}x_{i_1}.
\end{equation*}
Integrating this estimate against $(x_1, \ldots, x_{i_1-1}, x_{i_1+1}, \ldots, x_n)$, we have
\begin{equation*}
\int_{{\mathbb R}^n} M_{i_1}f(x)^p\cdot w(x){\rm d}x
\lesssim
\int_{{\mathbb R}^n} |f(x)|^p\cdot M_{i_1}w(x){\rm d}x.
\end{equation*}
Suppose that the result holds for $k-1$. Then,
fix $(x_1, \ldots, x_{i_k-1}, x_{i_k+1}, \ldots, x_n)$. Again, by the dual inequality of Stein type for $M_{i_k}$, 
we get
\begin{equation*}
\int_{{\mathbb R}} M_{i_k}\cdots M_{i_1}f(x)^p\cdot w(x){\rm d}x_{i_k}
\lesssim
\int_{{\mathbb R}} M_{i_{k-1}}\cdots M_{i_1}f(x)^p\cdot M_{i_k}w(x){\rm d}x_{i_k}.
\end{equation*}
Integrating this estimate against $(x_1, \ldots, x_{i_k-1}, x_{i_k+1}, \ldots, x_n)$ and using induction hypothesis, we have
\begin{align*}
\int_{{\mathbb R}^n} M_{i_k}\cdots M_{i_1}f(x)^p\cdot w(x){\rm d}x
&\lesssim
\int_{{\mathbb R}^n} M_{i_{k-1}}\cdots M_{i_1}f(x)^p\cdot M_{i_k}w(x){\rm d}x\\
&\lesssim
\int_{{\mathbb R}^n} |f(x)|^p\cdot M_{i_1}\cdots M_{i_k}w(x){\rm d}x.
\end{align*}
\end{proof}

The following corollary extends the dual inequality of Stein type for $L^p(\mathbb{R}^n)$.

\begin{corollary}\label{cor 171125-3}
Let $f,w$
be measurable functions.
Suppose in addition that $w \ge 0$
almost everywhere.
Then for all
$0<p<\infty$ and $0<t<p$, 
\begin{eqnarray*}
\int_{{\mathbb R}^n} \mathcal{M}_tf(x)^p\cdot w(x){\rm d}x
\lesssim
\int_{{\mathbb R}^n} |f(x)|^p\cdot  M_1\cdots M_nw(x){\rm d}x
\quad (f \in L^0(\mathbb{R}^n)).
\end{eqnarray*}
\end{corollary}

\begin{proof}
Using Proposition \ref{prop 171125-1}, we have
\begin{align*}
\int_{{\mathbb R}^n} \mathcal{M}_tf(x)^p\cdot w(x){\rm d}x
&=
\int_{{\mathbb R}^n} \left( M_n \cdots M_1 [|f|^t](x) \right)^{\frac{p}{t}}\cdot w(x){\rm d}x\\
&\lesssim
\int_{{\mathbb R}^n} (|f(x)|^t)^{\frac{p}{t}}\cdot M_1 \cdots M_n w(x) {\rm d}x\\
&=
\int_{{\mathbb R}^n} |f(x)|^p\cdot  M_1 \cdots M_nw(x){\rm d}x.
\end{align*}
\end{proof}

We recall Theorem \ref{thm 171115-1}.

\begin{theorem}[Dual inequality of Stein type for $L^{\vec{p}}$]
Let $f$ be a measurable function on $\mathbb{R}^n$ and $1\le\vec{p}<\infty$.
Then 
if $0< t<\min(p_1, \ldots, p_n)$ and $w_j^t\in A_{p_j}(\mathbb{R})$,
\begin{eqnarray*}
\left\|
{\mathcal M}_t f\cdot \bigotimes_{j=1}^n(w_j)^{\frac{1}{p_j}}
\right\|_{\vec{p}}
\lesssim
\left\|
f \cdot \bigotimes_{j=1}^n\left(M_j w_j\right)^{\frac{1}{p_j}}
\right\|_{\vec{p}}.
\end{eqnarray*}
\end{theorem}

\begin{proof}
We have only to check when $t=1$ and $1<\vec{p}<\infty$.
We use induction on $n$.
Let $n=1$. 
Then the result follows from the classical case of the dual inequality of Stein type.
Suppose that the result holds for $n-1$.
Then, the following inequality follows:
for $1<(q_1, \ldots, q_{n-1})<\infty$, $h \in L^{0}(\mathbb{R}^{n-1})$ and $v_j \in L^{0}(\mathbb{R})$,
\begin{equation} \label{eq 171115-2}
\left\|
(M_{n-1}\cdots M_1 h)\cdot 
\bigotimes_{j=1}^{n-1}v_j^{\frac{1}{q_j}}
\right\|_{(q_1, \ldots, q_{n-1})}
\lesssim
\left\|
h \cdot \bigotimes_{j=1}^{n-1}\left(M_j v_j\right)^{\frac{1}{q_j}}
\right\|_{(q_1, \ldots, q_{n-1})}.
\end{equation}
From the definition of the norm $\|\cdot\|_{\vec{p}}$, we get
\begin{align}\label{eq 180518-1}
&\left\|
(M_n\cdots M_1 f)\cdot 
\bigotimes_{j=1}^{n}w_j^{\frac{1}{p_j}}
\right\|_{\vec{p}}
=
\left\|
\left[
\left\|
(M_n\cdots M_1 f)\cdot 
\bigotimes_{j=1}^{n}w_j^{\frac{1}{p_j}}
\right\|_{(p_1, \ldots, p_{n-1})}
\right]
\right\|_{(p_n)} \nonumber\\
&=
\left\|
\left[
\left\|
(M_n\cdots M_1 f)\cdot 
\bigotimes_{j=1}^{n-1}w_j^{\frac{1}{p_j}}
\right\|_{(p_1, \ldots, p_{n-1})}
\right]w_n(\cdot_n)^\frac{1}{p_n}
\right\|_{(p_n)}.
\end{align}
By the Lebesgue differention theorem, $w_n\le M_nw_n$. Thus, by Lemma \ref{lem 180119-4},
\begin{align*}
&\left\|
(M_n\cdots M_1 f)\cdot 
\bigotimes_{j=1}^{n}w_j^{\frac{1}{p_j}}
\right\|_{\vec{p}}\\
&=
\left\|
\left[
\left\|
M_n\left[(M_{n-1}\cdots M_1 f)\cdot 
\bigotimes_{j=1}^{n-1}w_j^{\frac{1}{p_j}}\right]
\right\|_{(p_1, \ldots, p_{n-1})}
\right]w_n(\cdot_n)^\frac{1}{p_n}
\right\|_{(p_n)}\\
&\lesssim
\left\|
\left[
\left\|
(M_{n-1}\cdots M_1 f)\cdot 
\bigotimes_{j=1}^{n-1}w_j^{\frac{1}{p_j}}
\right\|_{(p_1, \ldots, p_{n-1})}
\right]w_n(\cdot_n)^\frac{1}{p_n}
\right\|_{(p_n)}\\
&\le
\left\|
\left[
\left\|
(M_{n-1}\cdots M_1 f)\cdot 
\bigotimes_{j=1}^{n-1}w_j^{\frac{1}{p_j}}
\right\|_{(p_1, \ldots, p_{n-1})}
\right]M_nw_n(\cdot_n)^\frac{1}{p_n}
\right\|_{(p_n)}.\\
\end{align*}
Thus, by induction hypothesis (\ref{eq 171115-2}),
\begin{align*}
&\left\|
(M_n\cdots M_1 f)\cdot 
\bigotimes_{j=1}^{n}w_j^{\frac{1}{p_j}}
\right\|_{\vec{p}}\\
&\lesssim
\left\|
\left[
\left\|
(M_{n-1}\cdots M_1 f)\cdot 
\bigotimes_{j=1}^{n-1}w_j^{\frac{1}{p_j}}
\right\|_{(p_1, \ldots, p_{n-1})}
\right]M_nw_n(\cdot_n)^\frac{1}{p_n}
\right\|_{(p_n)}\\
&\lesssim
\left\|
\left[
\left\|f\cdot 
\bigotimes_{j=1}^{n-1}(M_jw_j)^{\frac{1}{p_j}}
\right\|_{(p_1, \ldots, p_{n-1})}
\right]M_nw_n(\cdot_n)^\frac{1}{p_n}
\right\|_{(p_n)}\\
&=
\left\|
\left[
\left\|
f\cdot \bigotimes_{j=1}^{n}(M_jw_j)^{\frac{1}{p_j}}
\right\|_{(p_1, \ldots, p_{n-1})}
\right]
\right\|_{(p_n)}
=
\left\|
f \cdot \bigotimes_{j=1}^n\left(M_j w_j\right)^{\frac{1}{p_j}}
\right\|_{\vec{p}}.
\end{align*}
Thus, we conclude the result. 
\end{proof}

\section{Proof of Theorems \ref{thm 171123-2} and \ref{thm 171219-3}}
\label{sec vector-valued}

In this section, we shall prove the Fefferman--Stein vector-valued maximal inequality for mixed spaces.
First, we define the mixed vector-valued norm and show its duality formula.

\begin{definition}[Mixed vector-valued norm]
Let $0<\vec{p}\le\infty$ and $0<q \le\infty$.
For a system
$\displaystyle \{f_j \}_{j=1}^{\infty} \subset L^0({\mathbb R}^n)$,
define
\[
\| f_j \|_{L^{\vec{p}}(\ell^q)}
\equiv
\| \{f_j \}_{j=1}^{\infty} \|_{L^{\vec{p}}(\ell^q)}
=
\left\| \left(\sum_{j=1}^{\infty} |f_j|^q\right)^{\frac1q}
 \, \right\|_{\vec{p}}. 
\]
The space $L^{\vec{p}}(\ell^q,{\mathbb R}^n)$ denotes the set
of all collections $\{f_j\}_{j=1}^\infty$
for which the quantity
$\|\{f_j \}_{j=1}^{\infty} \|_{L^{\vec{p}}(\ell^q)}$
is finite.
\index{L p l q R n@$L^p(\ell^q,{\mathbb R}^n)$}
A natural modification is made in the above when $q=\infty$.
\index{vector-valued norm@vector-valued norm}
\end{definition}

This vector-valued norm can be written by the form of duality.

\begin{lemma}
\label{lem 171123-1}
Let $1<\vec{p}\le\infty$ and $1<q \le\infty$,
and let
$\{f_j\}_{j=1}^\infty$
be a sequence of $L^{\vec{p}}({\mathbb R}^n)$-functions
such that
$f_j=0$ a.e.
if $j$ is large enough.
Then we can take
a sequence
$\{g_j\}_{j=1}^\infty$ of $L^{\vec{p'}}({\mathbb R}^n)$-functions
such that
\begin{eqnarray*}
\| f_j \|_{L^{\vec{p}}(\ell^q)}
=
\sum_{j=1}^\infty
\int_{{\mathbb R}^n} f_j(x)g_j(x){\rm d}x, \quad
\| g_j \|_{L^{\vec{p'}}(\ell^{q'})}
=1.
\end{eqnarray*}
If
$\{f_j\}_{j=1}^\infty$
is nonnegative,
then we can arrange
that
$\{g_j\}_{j=1}^\infty$
is nonnegative.
\end{lemma}

\begin{proof}
There is nothing to prove if
$f_j(x)=0$
for all nonnegative integers $j$
and
for almost all $x \in {\mathbb R}^n$;
assume otherwise.
In this case,
we recall the construction of the duality
$L^{\vec{p}}({\mathbb R}^n)$-$L^{\vec{p'}}({\mathbb R}^n)$ \cite{B-P};
for $x \in \mathbb{R}^n$, set
\[
g_j(x)\equiv \overline{{\rm sgn}(f_j)(x)}\,|f_j(x)|^{q-1}\,
\left(\sum_{j=1}^\infty |f_j(x)|^q\right)^{\frac{p_1}{q}-1}
\prod_{k=1}^n
\left\|
\left(\sum_{j=1}^\infty |f_j|^q\right)
\right\|_{(p_1,\ldots,p_k)}^{p_{k+1}-p_k}(x'),
\]
where we let $p_{n+1}=1$ and $x'=(x_{k+1}, \ldots, x_n)$. Since
\[
\left(
\sum_{j=1}^\infty |g_j|^{q'}
\right)^{\frac{p'_1}{q'}}
=
\left(
\sum_{j=1}^\infty |f_j|^q
\right)^{\frac{p_1}{q}}
\prod_{k=1}^n
\left\|\left(\sum_{j=1}^\infty |f_j|^q\right)^{\frac{1}{q}}\right\|_{(p_1, \ldots, p_k)}^{p_{k+1}-p_k},
\]
we have
$\| g_j \|_{L^{\vec{p'}}(\ell^{q'})}=1$.
Furthermore, since
\[
\sum_{j=1}^\infty f_jg_j
=
\left(
\sum_{j=1}^\infty |f_j|^q
\right)^{\frac{p_1}{q}}
\prod_{k=1}^n
\left\|
\left(\sum_{j=1}^\infty |f_j|^q\right)^{\frac{1}{q}}
\right\|_{(p_1,\ldots,p_k)}^{p_{k+1}-p_k},
\]
we obtain
$\displaystyle
\| f_j \|_{L^{\vec{p}}(\ell^q)}
=
\sum_{j=1}^\infty
\int_{{\mathbb R}^n}f_j(x)g_j(x){\rm d}x.
$
\end{proof}




To prove the Fefferman--Stein vector-valued maximal inequality for $L^{\vec{p}}(\mathbb{R}^n)$,
we use the following lemma, which was proved by Bagby \cite{Bagby}. 
This lemma is the unweighted version of Lemma \ref{lem 180115-1}.

\begin{lemma}{\rm (\cite{Bagby})} \label{lem 180114-2}
Let $\{f_{(j_1, \ldots, j_m)}\}_{j_1, \ldots, j_m=1}^\infty \subset L^0(\mathbb{R}^n)$.
For $1<q_i<\infty (i=1, \ldots,m)$ and $1<p<\infty$, 
\[
\left\|\left[\left\|Mf_{(j_1,\ldots,j_m)}\right\|_{\ell^{(q_1,\ldots,q_m)}}\right]\right\|_{p}
\lesssim
\left\|\left[\left\|f_{(j_1,\ldots,j_m)}\right\|_{\ell^{(q_1,\ldots,q_m)}}\right]\right\|_{p},
\]
that is,
\begin{eqnarray*}
&&
\left\|
\left(
\sum_{j_m=1}^\infty
\left(\cdots
\sum_{j_2=1}^\infty
\left(
\sum_{j_1=1}^\infty
(M f_{(j_1,\ldots,j_m)})^{q_1}
\right)^{\frac{q_2}{q_1}}
\cdots
\right)^{\frac{q_n}{q_{m-1}}}
\right)^{\frac{1}{q_m}}\right\|_{p}\\
&&\lesssim
\left\|\left(
\sum_{j_m=1}^\infty
\left(\cdots
\sum_{j_2=1}^\infty
\left(
\sum_{j_1=1}^\infty
|f_{(j_1,\ldots,j_m)}|^{q_1}
\right)^{\frac{q_2}{q_1}}
\cdots
\right)^{\frac{q_m}{q_{m-1}}}
\right)^{\frac{1}{q_m}}\right\|_{p}.
\end{eqnarray*}
 
\end{lemma}

\begin{theorem}[Fefferman--Stein vector-valued maximal inequality] \label{thm 180106-1}
Let $0<\vec{p}<\infty$, $0<u\le\infty$ and $0<t<\min(p_1, \ldots, p_n, u)$. Then,
for $\{f_k\}_{k=1}^\infty \subset L^0(\mathbb{R}^n)$,
\begin{equation*}
\left\|\left(
\sum_{k=1}^{\infty}
[\mathcal{M}_tf_k]^u
\right)^{\frac{1}{u}}\right\|_{\vec{p}}
\lesssim
\left\|\left(
\sum_{k=1}^{\infty}
|f_k|^u
\right)^{\frac{1}{u}}\right\|_{\vec{p}}.
\end{equation*}
\end{theorem}
\begin{proof}
As we did in Theorem \ref{thm 171206-1},
we can reduce the matters to the case $t=1$ and $1<\vec{p}<\infty$.
\begin{itemize}


\item[(i)]
Let $1<u<\infty$.
We may assume that $f_k=0$ for $k \gg 1$,
so that at least we know that both sides are finite since
we already showed that ${\mathcal M}_1$ is $L^{\vec{p}}$-bounded.
We induct on $n$.
If $n=1$, then this is nothing but the Fefferman--Stein vector-valued inequality.
Assume that for all $\{g_k\}_{k=1}^\infty \subset L^0({\mathbb R}^{n-1})$
\[
\left\|\left(
\sum_{k=1}^\infty \left[M_{n-1}\cdots M_1g_k\right]^u
\right)^{\frac1u}\right\|_{(p_1, \ldots, p_{n-1})} \lesssim
\left\|\left(\sum_{k=1}^\infty |g_k|^u\right)^{\frac1u}\right\|_{(p_1, \ldots, p_{n-1})}.
\]

Assume that
$f \in L^{\vec{p}}({\mathbb R}^n)$ is a function of the form:
\[
f_k(x_1,\ldots,x_n)
=
\sum_{m' \in {\mathbb Z}^{n-1}}
\chi_{m'+[0,1]^{n-1}}(r x')f_{k,m'}(x_n),
\]
where $r>0$ and $\{f_{k,m'}\}_{m' \in {\mathbb Z}^{n-1}} \subset L^0({\mathbb R})$.
Then
\[
M_n f_k(x_1,\ldots,x_n)
=
\sum_{m' \in {\mathbb Z}^{n-1}}
\chi_{m'+[0,1]^{n-1}}(r x')M_n f_{k,m'}(x_n),
\]
since the summand is made up of at most one non-zero function
once we fix $x'$.
Define $v>0$ by
\[
\frac{1}{v}=\frac{1}{p_1}+\cdots+\frac{1}{p_{n-1}}.
\]
We observe
\begin{align*}
\left\| \left\{ M_n f_k \right\}_{k=1}^{\infty} \right\|_{L^{\vec{p}}(\ell^u)}
&=
\left\|
\left(\sum_{k=1}^{\infty}
\left[
\sum_{m' \in {\mathbb Z}^{n-1}}
\chi_{m'+[0,1]^{n-1}}(r \cdot')M_nf_{k,m'}(\cdot_n)
\right]^u\right)^\frac1u
\right\|_{\vec{p}}\\
&=
r^{-\frac{1}{v}}
\left\|
\left(\sum_{k=1}^{\infty}
\left[
\sum_{m' \in {\mathbb Z}^{n-1}}
\chi_{m'+[0,1]^{n-1}}(\cdot')M_nf_{k,m'}(\cdot_n)
\right]^u\right)^\frac1u
\right\|_{\vec{p}}.\\
\end{align*}
Setting $\vec{s}=(p_1, \ldots,p_{n-1})$, we get
\begin{align*}
&\left\|
\left(\sum_{k=1}^\infty\left[
\sum_{m' \in {\mathbb Z}^{n-1}}
\chi_{m'+[0,1]^{n-1}}(\cdot')M_nf_{k,m'}(\cdot_n)\right]^u\right)^\frac1u
\right\|_{\vec{p}}\\
&=
\left\|\left\|\left(\sum_{k=1}^\infty
\left[
\sum_{m' \in {\mathbb Z}^{n-1}}
\chi_{m'+[0,1]^{n-1}}(\cdot')M_nf_{k,m'}(\cdot_n)
\right]^u\right)^\frac1u\right\|_{\vec{s}}\right\|_{(p_n)}\\
&=
\left\|
\left[
\left\|
\left(\sum_{k=1}^\infty
\left[
M_nf_{k,m'}(\cdot_n)
\right]^u\right)^\frac1u
\right\|_{\ell^{(p_1, \ldots, p_{n-1})}}
\right]
\right\|_{(p_n)}\\
&=
\left\|
\left[
\left\|
M_nf_{k,m'}(\cdot_n)
\right\|_{\ell^{(u, p_1, \ldots, p_{n-1})}}
\right]
\right\|_{(p_n)}.
\end{align*}
Thus by Lemma \ref{lem 180114-2},
we obtain
\begin{align*}
\left\|\left(\sum_{k=1}^\infty
[M_nf_{k}]^u\right)^\frac1u\right\|_{\vec{p}}
&\lesssim
r^{-\frac{1}{v}}
\left\|
\left[
\left\|
M_nf_{k,m'}(\cdot_n)
\right\|_{\ell^{(u, p_1, \ldots, p_{n-1})}}
\right]
\right\|_{(p_n)}\\
&\lesssim
r^{-\frac{1}{v}}
\left\|
\left[
\left\|
f_{k,m'}(\cdot_n)
\right\|_{\ell^{(u, p_1, \ldots, p_{n-1})}}
\right]
\right\|_{(p_n)}\\
&=
\left\|
\left[
\left\|
\left(\sum_{k=1}^\infty
\left[
f_{k,m'}(\cdot_n)
\right]^u\right)^\frac1u
\right\|_{\ell^{(p_1, \ldots, p_{n-1})}}
\right]
\right\|_{(p_n)}\\
&=
\left\|
\left(\sum_{k=1}^\infty
|f_{k}|^u\right)^\frac1u
\right\|_{\vec{p}}.\\
\end{align*}
Here the constant is independent of $r>0$.
Let $f \in L^{\vec{p}}({\mathbb R}^n)$ be arbitrary.
Write
\[
f_k^{(r)}(x)=
\frac{1}{r^{n-1}}
\sum_{m' \in {\mathbb Z}^{n-1}}
\chi_{r m'+[0,r]^{n-1}}(x')
\int_{r m'+[0,r]^{n-1}}f_k(y',x_n)\,dy'.
\]
Thanks to the Lebesgue differentiation theorem,
\[
f_k(x', x_n)=\lim_{r \downarrow 0}f_k^{(r)}(x', x_n)
\]
for almost every $x' \in {\mathbb R}^{n-1}$.
Thus, by the Fatou lemma, we obtain
\[
M_n f_k(x) \le \liminf_{r \downarrow 0}M_n f_k^{(r)}(x).
\]
Meanwhile, for all $r>0$, since $f_k^{(r)}\le M_{n-1}\cdots M_{1}f_k$,
by induction assumption, 
\begin{align*}
\left\|\left(\sum_{k=1}^\infty|f_k^{(r)}|^u\right)^\frac1u\right\|_{\vec{p}}
\le
\left\|\left(\sum_{k=1}^\infty[M_{n-1}\cdots M_{1}f_k]^u\right)^\frac1u\right\|_{\vec{p}}
\lesssim
\left\|\left(\sum_{k=1}^\infty|f_k|^u\right)^\frac1u\right\|_{\vec{p}},
\end{align*}
where $\vec{s}=(p_1,\ldots, p_{n-1})$.
As a consequence,
by the Lebesgue differentiation theorem and the Fatou lemma,
we obtain
\begin{align*}
\left\|\left(\sum_{k=1}^\infty[M_n f_k]^u\right)^\frac1u\right\|_{\vec{p}}
&\le\liminf_{r \downarrow 0}
\left\|\left(\sum_{k=1}^\infty[M_n f_k^{(r)}]^u\right)^\frac1u\right\|_{\vec{p}}\\
&\lesssim\liminf_{r \downarrow 0}
\left\|\left(\sum_{k=1}^\infty|f_k^{(r)}|^u\right)^\frac1u\right\|_{\vec{p}}
\le
\left\|\left(\sum_{k=1}^\infty|f_k|^u\right)^\frac1u\right\|_{\vec{p}}.
\end{align*}
\end{itemize}
Therefore, by induction assumption,
\begin{align*}
\left\|\left(\sum_{k=1}^\infty[M_nM_{n-1}\cdots M_1 f_k]^u\right)^\frac1u\right\|_{\vec{p}}
&=
\left\|\left(\sum_{k=1}^\infty[M_n(M_{n-1}\cdots M_1 f_k)]^u\right)^\frac1u\right\|_{\vec{p}}\\
&\lesssim
\left\|\left(\sum_{k=1}^\infty[M_{n-1}\cdots M_1 f_k]^u\right)^\frac1u\right\|_{\vec{p}}\\
&\lesssim
\left\|\left(\sum_{k=1}^\infty|f_k|^u\right)^\frac1u\right\|_{\vec{p}}.
\end{align*}
\item[(ii)]
Let $u=\infty$. Then, simply using
\[
\sup_{k\in \mathbb{N}} \mathcal{M}_1f_k
\le
\mathcal{M}_1\left[\sup_{k\in \mathbb{N}} f_k\right],
\]
we get the result.
\end{proof}



We can also show the vector-valued inequality for the Hardy--Littlewood maximal operator 
in mixed Morrey spaces. 

\begin{theorem} 
Let $1<\vec{q}<\infty$, $1<u\le\infty$, and $1<p\le\infty$ satisfy $\frac np \le\sum_{j=1}^n\frac1q_j$.
Then, for every sequence $\{f_j\}_{j=1}^{\infty} \in L^0(\mathbb{R}^n)$,
\begin{equation*}
\left\|\left(
\sum_{j=1}^{\infty}
[Mf_j]^u
\right)^{\frac{1}{u}}\right\|_{\mathcal{M}^p_{\vec{q}}(\mathbb{R}^n)}
\lesssim
\left\|\left(
\sum_{j=1}^{\infty}
|f_j|^u
\right)^{\frac{1}{u}}\right\|_{\mathcal{M}^p_{\vec{q}}(\mathbb{R}^n)}.
\end{equation*}
\end{theorem}

\begin{proof}
\begin{itemize}
\item[(i)]
Let $u=\infty$. Then, simply using
\[
\sup_{j\in \mathbb{N}} Mf_j
\le
M\left[\sup_{j\in \mathbb{N}} f_j\right],
\]
we get the result.

\item[(ii)]
Let $1<u<\infty$.
We have to show that
\[
|Q|^{\frac{1}{p}-\frac{1}{n}\left(\sum_{j=1}^n\frac{1}{q_j}\right)}
\left\|\left(
\sum_{j=1}^{\infty}
[Mf_j]^u
\right)^{\frac{1}{u}}\chi_Q\right\|_{\vec{q}}
\lesssim
\left\|\left(
\sum_{j=1}^{\infty}
|f_j|^u
\right)^{\frac{1}{u}}\right\|_{\mathcal{M}^p_{\vec{q}}(\mathbb{R}^n)}.
\]
Let $f_{j,1}=f_j\chi_{5Q}$ and $f_{j,2}=f_j-f_{j,1}$.
Using subadditivity of $M$, we have
\begin{align*}
\left\|\left(
\sum_{j=1}^{\infty}
[Mf_j]^u
\right)^{\frac{1}{u}}\chi_Q\right\|_{\vec{q}}
&\le
\left\|\left(
\sum_{j=1}^{\infty}
[Mf_{j,1}]^u
\right)^{\frac{1}{u}}\chi_Q\right\|_{\vec{q}}
+\left\|\left(
\sum_{j=1}^{\infty}
[Mf_{j,2}]^u
\right)^{\frac{1}{u}}\chi_Q\right\|_{\vec{q}}\\
&\equiv
J_1+J_2.
\end{align*}
First, using Theorem \ref{thm 171123-2}, we have
\begin{align*}
|Q|^{\frac{1}{p}-\frac{1}{n}\left(\sum_{j=1}^n\frac{1}{q_j}\right)}J_1
&\le
|Q|^{\frac{1}{p}-\frac{1}{n}\left(\sum_{j=1}^n\frac{1}{q_j}\right)}
\left\|\left(
\sum_{j=1}^{\infty}
[Mf_{j,1}]^u
\right)^{\frac{1}{u}}\right\|_{\vec{q}}\\
&\lesssim
|Q|^{\frac{1}{p}-\frac{1}{n}\left(\sum_{j=1}^n\frac{1}{q_j}\right)}
\left\|\left(
\sum_{j=1}^{\infty}
|f_{j,1}|^u
\right)^{\frac{1}{u}}\right\|_{\vec{q}}\\
&=
|Q|^{\frac{1}{p}-\frac{1}{n}\left(\sum_{j=1}^n\frac{1}{q_j}\right)}
\left\|\left(
\sum_{j=1}^{\infty}
|f_j|^u
\right)^{\frac{1}{u}}\chi_{5Q}\right\|_{\vec{q}}\\
&\lesssim
\left\|\left(
\sum_{j=1}^{\infty}
|f_j|^u
\right)^{\frac{1}{u}}\right\|_{\mathcal{M}^p_{\vec{q}}(\mathbb{R}^n)}.
\end{align*} 
Second, let $y \in Q$. By Proposition \ref{ex 171103-1},
\[
Mf_{j,2}(y) \lesssim \sup_{Q\subset R} \frac{1}{|R|}\int_{R}|f_j(y)|{\rm d}y
\lesssim
\sup_{\ell \in \mathbb{N}}\frac{1}{|2^{\ell}Q|}\int_{2^{\ell}Q}|f_j(y)|{\rm d}y.
\] 
We decompose
\[
2^{\ell}Q=\bigcup_{k=1}^{2^{\ell}}Q^{(k)}, \quad |Q^{(k)}|=|Q|.
\]
Thus,
\[
Mf_{j,2}(y) \lesssim
\sup_{\ell \in \mathbb{N}}\sum_{j=1}^{2^{\ell}}\frac{1}{|2^{\ell}Q|}\int_{Q^{(k)}}|f_j(y)|{\rm d}y
\le
\sup_{\ell \in \mathbb{N}}\max_{k=1, \ldots, 2^{\ell}}\frac{1}{|Q^{(k)}|}\int_{Q^{(k)}}|f_j(y)|{\rm d}y.
\]
Using Minkowski's inequality, we get
\begin{align*}
\left(\sum_{j=1}^{\infty}Mf_{j,2}(y)^u\right)^{\frac{1}{u}}
&\lesssim
\sup_{\ell \in \mathbb{N}}\max_{k=1, \ldots, 2^{\ell}}\frac{1}{|Q^{(k)}|}
\left[\sum_{j=1}^{\infty}\left(\int_{Q^{(k)}}|f_j(y)|{\rm d}y\right)^u\right]^{\frac{1}{u}}\\
&\le
\sup_{\ell \in \mathbb{N}}\max_{k=1, \ldots, 2^{\ell}}\frac{1}{|Q^{(k)}|}
\int_{Q^{(k)}}\left(\sum_{j=1}^{\infty}|f_j(y)|^u\right)^{\frac{1}{u}}{\rm d}y.\\
\end{align*}
Multiplying $\chi_Q$ and taking $L^{\vec{q}}$-norm, we have
\[
\left\|\left(\sum_{j=1}^{\infty}(Mf_{j,2})^u\right)^{\frac{1}{u}}\chi_Q\right\|_{\vec{q}}
\lesssim
\sup_{\ell \in \mathbb{N}}\max_{k=1, \ldots, 2^{\ell}}\frac{1}{|Q^{(k)}|}
\int_{Q^{(k)}}\left(\sum_{j=1}^{\infty}|f_j(y)|^u\right)^{\frac{1}{u}}{\rm d}y \times \|\chi_Q\|_{\vec{q}}.
\]
Therefore, using relation 
${\mathcal M}^p_{\vec{q}}({\mathbb R}^n) 
\hookrightarrow {\mathcal M}^p_{(\underbrace{1, \ldots, 1}_{\mbox{$n$ times}})}({\mathbb R}^n)
= {\mathcal M}^p_1({\mathbb R}^n)$, we obtain
\begin{align*}
&|Q|^{\frac{1}{p}-\frac{1}{n}\left(\sum_{j=1}^n\frac{1}{q_j}\right)}J_2\\
&\lesssim
|Q|^{\frac{1}{p}-\frac{1}{n}\left(\sum_{j=1}^n\frac{1}{q_j}\right)}
\sup_{\ell \in \mathbb{N}}\max_{k=1, \ldots, 2^{\ell}}\frac{1}{|Q^{(k)}|}
\int_{Q^{(k)}}\left(\sum_{j=1}^{\infty}|f_j(y)|^u\right)^{\frac{1}{u}}{\rm d}y
\times |Q|^{\frac{1}{n}\left(\sum_{j=1}^n\frac{1}{q_j}\right)}\\
&=
\sup_{\ell \in \mathbb{N}}\max_{k=1, \ldots, 2^{\ell}}
|Q^{(k)}|^{\frac{1}{p}-1}\int_{Q^{(k)}}\left(\sum_{j=1}^{\infty}|f_j(y)|^u\right)^{\frac{1}{u}}{\rm d}y\\
&\le
\left\|\left(\sum_{j=1}^{\infty}|f_j|^u\right)^{\frac{1}{u}}\right\|_{\mathcal{M}^p_1(\mathbb{R}^n)}
\le
\left\|\left(\sum_{j=1}^{\infty}|f_j|^u\right)^{\frac{1}{u}}\right\|_{\mathcal{M}^p_{\vec{q}}(\mathbb{R}^n)}.
\end{align*}
\end{itemize}
Thus, the result holds.
\end{proof}

We can also prove the Fefferman--Stein vector-valued inequality for the iterated maximal operator in mixed Morrey spaces.
The way is similar to Theorem \ref{thm 180119-1}.
First, we prepare the following proposition, which is vector-valued case for Proposition \ref{prop 180119-6}.

\begin{proposition}\label{prop 180129-5}
Let $1<\vec{q}<\infty$ and $w_k \in A_{q_k}(\mathbb{R})$ for $k=1,\ldots,n$.
Then, 
\[
\left\|
\left(\sum_{j=1}^\infty
[\mathcal{M}_1f_j]^u
\right)^\frac1u
\cdot\bigotimes_{k=1}^nw_{k}^{\frac{1}{q_k}}
\right\|_{\vec{q}}
\lesssim
\left\|
\left(\sum_{j=1}^\infty
|f_j|^u
\right)^\frac1u
\cdot\bigotimes_{k=1}^nw_{k}^{\frac{1}{q_k}}
\right\|_{\vec{q}},
\]
for $f \in L^0(\mathbb{R}^n)$.
\end{proposition}

\begin{proof}
We induct on $n$.
Let $n=1$. Then, this is clear by Lemma \ref{lem 180115-1}. 
Suppose that the result holds for $n-1$, that is, 
\begin{align*}
\left\|
\left(\sum_{j=1}^\infty
[M_{n-1}\cdots M_1h_j]^u
\right)^\frac1u
\cdot
\bigotimes_{k=1}^{n-1}w_{k}^{\frac{1}{q_k}}
\right\|_{(q_1, \ldots, q_{n-1})}\\
\lesssim
\left\|
\left(\sum_{j=1}^\infty
|h_j|^u
\right)^\frac1u
\cdot\bigotimes_{k=1}^{n-1}w_{k}^{\frac{1}{q_k}}
\right\|_{(q_1, \ldots, q_{n-1})},
\end{align*}
for $h_j \in L^0(\mathbb{R}^{n-1})$.
Then, again by Lemma \ref{lem 180119-4}, 
\begin{align}\label{eq 180517-2}
&\left\|
\left(\sum_{j=1}^\infty
[\mathcal{M}_1f_j]^u
\right)^\frac1u
\cdot\bigotimes_{k=1}^nw_{k}^{\frac{1}{q_k}}
\right\|_{\vec{q}}\nonumber \\
&\lesssim
\left\|
\left[\left\|
\left(\sum_{j=1}^\infty
[M_{n-1}\cdots M_1f_j]^u
\right)^\frac1u \cdot\bigotimes_{k=1}^{n-1}w_{k}^{\frac{1}{q_k}}
\right\|_{(q_1, \ldots, q_{n-1})}\right]
w_n(\cdot_n)^\frac{1}{q_n}
\right\|_{(q_n)}.
\end{align}
Thus, by induction hypothesis, 
\begin{align*}
&\mbox{the right-hand side of (\ref{eq 180517-2})}\\
&=
\left\|
\left[\left\|
\left(\sum_{j=1}^\infty
[M_{n-1}\cdots M_1f_j]^u
\right)^\frac1u 
\cdot\bigotimes_{k=1}^{n-1}w_{k}^{\frac{1}{q_k}}
\right\|_{(q_1, \ldots, q_{n-1})}\right]
w_n(\cdot_n)^\frac{1}{q_n}\right\|_{(q_n)}\\
&\lesssim
\left\|\left[
\left\|
\left(\sum_{j=1}^\infty
|f_j|^u
\right)^\frac1u
\cdot\bigotimes_{k=1}^{n-1}w_{k}^{\frac{1}{q_k}}
\right\|_{(q_1, \ldots, q_{n-1})}\right]
w_n(\cdot_n)^\frac{1}{q_n}\right\|_{(q_n)}\\
&=
\left\|
\left(\sum_{j=1}^\infty
|f_j|^u
\right)^\frac1u
\cdot\bigotimes_{k=1}^{n}w_{k}^{\frac{1}{q_k}}
\right\|_{\vec{q}}.
\end{align*}
\end{proof}

\begin{theorem} \label{thm 180129-1}
Let $0<\vec{q}\le\infty$ and $0<p<\infty$ satisfy 
\[
\frac np \le\sum_{j=1}^n\frac1q_j,
\quad
\frac{n-1}{n}p<\max(q_1,\ldots,q_n).
\]
 If $0<t<\min(q_1, \ldots, q_n, p)$, then
\begin{equation*}
\left\|\left(\sum_{j=1}^\infty
[{\mathcal M}_tf_j]^u\right)^\frac1u
\right\|_{\mathcal{M}^p_{\vec{q}}(\mathbb{R}^n)} 
\lesssim \left\|\left(\sum_{k=1}^\infty|f_j|^u\right)^\frac1u\right\|_{\mathcal{M}^p_{\vec{q}}(\mathbb{R}^n)}
\end{equation*}
for all $f \in \mathcal{M}^p_{\vec{q}}(\mathbb{R}^n)$.
\end{theorem}

\begin{proof}
We have only to check for $t=1, 1<p<\infty$ and $1<\vec{q}<\infty$
as we did in Theorem \ref{thm 171206-1}.
For $\eta \in \mathbb{R}$ satisfying
\begin{equation} \label{eq 180129-2}
0<\sum_{j=1}^n \frac{1}{q_j}-\frac{n}{p}<\eta<\frac{1}{\max(q_1,\ldots,q_n)},
\end{equation}
once we show  
\begin{equation}\label{eq 180129-3}
\left\|\left(\sum_{j=1}^\infty
[{\mathcal M}_1f_j]^u\right)^\frac1u(\mathcal{M}_1\chi_Q)^\eta\right\|_{\vec{q}}
\lesssim
\left\|\left(\sum_{k=1}^\infty|f_j|^u\right)^\frac1u(\mathcal{M}_1\chi_Q)^\eta\right\|_{\vec{q}},
\end{equation}
we get
\begin{align*}
|Q|^{\frac1p-\frac1n\sum_{j=1}^n \frac{1}{q_j}}
&\left\|\left(\sum_{j=1}^\infty
[{\mathcal M}_1f_j]^u\right)^\frac1u(\mathcal{M}_1\chi_Q)^\eta\right\|_{\vec{q}}\\
&\lesssim
|Q|^{\frac1p-\frac1n\sum_{j=1}^n \frac{1}{q_j}}
\left\|\left(\sum_{k=1}^\infty|f_j|^u\right)^\frac1u(\mathcal{M}_1\chi_Q)^\eta\right\|_{\vec{q}}.
\end{align*}
Taking supremum for all cubes and using Proposition \ref{prop 180119-5}, we conclude the result.

We shall show (\ref{eq 180129-3}).
Let $Q=I_1 \times I_2 \times \cdots \times I_n$. Then, 
\[
(\mathcal{M}_1\chi_Q)^\eta
=
\left(\bigotimes_{j=1}^nM_j\chi_{I_j}\right)^\eta
=
\bigotimes_{j=1}^n\left(M_j\chi_{I_j}\right)^\eta.
\]
Here, $(M_j\chi_{I_j})^{\eta q_j}$ is $A_1$-weight if and only if 
\begin{equation}\label{eq 180129-4}
0\le\eta q_j<1.
\end{equation} 

Since $(M_j\chi_{I_j})^{\eta q_j} \in A_1 \subset A_{q_j}$ for all $q_j$,
by Proposition \ref{prop 180129-5}, we obtain
\begin{align*}
\left\|\left(\sum_{j=1}^\infty
[{\mathcal M}_1f_j]^u\right)^\frac1u(\mathcal{M}_1\chi_Q)^\eta\right\|_{\vec{q}}
&=
\left\|\left(\sum_{j=1}^\infty
[{\mathcal M}_1f_j]^u\right)^\frac1u\bigotimes_{j=1}^n\left(M_j\chi_{I_j}\right)^\eta\right\|_{\vec{q}}\\
&\lesssim
\left\|\left(\sum_{k=1}^\infty|f_j|^u\right)^\frac1u\bigotimes_{j=1}^n\left(M_j\chi_{I_j}\right)^\eta\right\|_{\vec{q}}\\
&=
\left\|\left(\sum_{k=1}^\infty|f_j|^u\right)^\frac1u(\mathcal{M}_1\chi_Q)^\eta\right\|_{\vec{q}}.
\end{align*}
Thus, (\ref{eq 180129-3}) holds.

\end{proof}

\begin{corollary}
Let 
\begin{equation*}
0<\frac{n-1}{n}p<q\le p<\infty.
\end{equation*}
 If $0<t<q$, then
\[
\left\|\left(\sum_{j=1}^\infty
[{\mathcal M}_tf_j]^u\right)^\frac1u
\right\|_{\mathcal{M}^p_q(\mathbb{R}^n)} 
\lesssim 
\left\|
\left(\sum_{j=1}^\infty
|f_j|^u\right)^\frac1u
\right\|_{\mathcal{M}^p_q(\mathbb{R}^n)}
\]
for all $f \in \mathcal{M}^p_q(\mathbb{R}^n)$.

\end{corollary}

\begin{proof}
In Theorem \ref{thm 180129-1}, letting $q_j=q$, we conclude the result. 
\end{proof}

\section{Proof of Theorem \ref{thm 171219-1} and \ref{thm 180514-1}}
\label{sec fractional}

In the beginning of this section, we show the boundedness of the fractional integral operator.
We follow the idea of Tanaka \cite{Tanaka}.


\begin{proof}
Fix $x \in \mathbb{R}^n$.
Without loss of generality, we may asuume that $f$ is non-negative and $I_{\alpha}f(x)$ is finite.
Then, we see that there exists $R>0$ such that 
\[
\int_{\{|x-y|\le R\}}\frac{f(y)}{|x-y|^{n-\alpha}}{\rm d}y=\frac{I_{\alpha}f(x)}{2}.
\]
We shall obtain two estimates.
First,
\begin{align*}
\frac{I_{\alpha}f(x)}{2}
&=
\int_{\{|x-y|\le R\}}\frac{f(y)}{|x-y|^{n-\alpha}}{\rm d}y\\
&=
\sum_{j=-\infty}^0
\int_{\{2^{j-1}R<|x-y|\le 2^jR\}}\frac{f(y)}{|x-y|^{n-\alpha}}{\rm d}y\\
&\lesssim
\sum_{j=-\infty}^0
\frac{(2^jR)^{\alpha}}{(2^jR)^{n}}\int_{\{|x-y|\le 2^jR\}}f(y){\rm d}y\\
&\le
Mf(x)\sum_{j=-\infty}^0(2^jR)^{\alpha}\\
&\sim
R^{\alpha}Mf(x).
\end{align*}
Second,
\begin{align*}
\frac{I_{\alpha}f(x)}{2}
&=
\int_{\{|x-y|\le R\}}\frac{f(y)}{|x-y|^{n-\alpha}}{\rm d}y
=
\sum_{j=1}^{\infty}
\int_{\{2^{j-1}R<|x-y|\le 2^jR\}}\frac{f(y)}{|x-y|^{n-\alpha}}{\rm d}y
\end{align*}
Using Proposition \ref{prop 180521-1}, we get
\begin{align*}
\frac{I_{\alpha}f(x)}{2}
&\lesssim
\sum_{j=1}^{\infty}
\frac{(2^jR)^{\alpha}}{(2^jR)^{n}}\int_{\{|x-y|\le 2^jR\}}f(y){\rm d}y\\
&=
\sum_{j=1}^{\infty}
\frac{(2^jR)^{\alpha}}{(2^jR)^{\frac{n}{p}}}
(2^jR)^{n\left(\frac{1}{p}-1\right)}
\int_{\{|x-y|\le 2^jR\}}f(y){\rm d}y\\
&\le
\|f\|_{\mathcal{M}_1^p(\mathbb{R}^n)}
\sum_{j=1}^{\infty}
\frac{(2^jR)^{\alpha}}{(2^jR)^{\frac{n}{p}}}\\
&\sim
R^{\alpha-\frac{n}{p}}\|f\|_{\mathcal{M}_1^p(\mathbb{R}^n)}
\le
R^{\alpha-\frac{n}{p}}\|f\|_{\mathcal{M}_{\vec{q}}^p(\mathbb{R}^n)}
=
R^{-\frac{n}{r}}\|f\|_{\mathcal{M}_{\vec{q}}^p(\mathbb{R}^n)}.
\end{align*}
Thus, we obtain
\begin{align*}
I_{\alpha}f(x)
\lesssim
\min(R^{\alpha}Mf(x), R^{-\frac{n}{r}}\|f\|_{\mathcal{M}_{\vec{q}}^p(\mathbb{R}^n)}).
\end{align*}
We now delete the factor $R$ by the following argument:
\begin{align*}
I_{\alpha}f(x)
&\lesssim
\min(R^{\alpha}Mf(x), R^{-\frac{n}{r}}\|f\|_{\mathcal{M}_{\vec{q}}^p(\mathbb{R}^n)})\\
&\le
\sup_{t>0}\min(t^{\alpha}Mf(x), t^{-\frac{n}{r}}\|f\|_{\mathcal{M}_{\vec{q}}^p(\mathbb{R}^n)})\\
&=
\|f\|_{\mathcal{M}_{\vec{q}}^p(\mathbb{R}^n)}^{\frac{p\alpha}{n}}
Mf(x)^{1-\frac{p\alpha}{n}},
\end{align*}
where we use the condition $\frac{1}{r}=\frac{1}{p}-\frac{\alpha}{n}$.
It follows from the conditions 
\[
\frac{1}{r}=\frac{1}{p}-\frac{\alpha}{n}, 
\]
that
\[
1-\frac{p\alpha}{n}=\frac{p}{r}.
\]
Thus, we get 
\[
I_{\alpha}f(x)\lesssim \|f\|_{\mathcal{M}_{\vec{q}}^p(\mathbb{R}^n)}^{1-\frac{p}{r}}
Mf(x)^{\frac{p}{r}}.
\]
This pointwise estimate gives us that
\[
\|I_{\alpha}f\|_{\mathcal{M}_{\vec{s}}^r(\mathbb{R}^n)}
\lesssim \|f\|_{\mathcal{M}_{\vec{q}}^p(\mathbb{R}^n)}^{1-\frac{p}{r}}
\|\left[Mf\right]^{\frac{p}{r}}\|_{\mathcal{M}_{\vec{s}}^r(\mathbb{R}^n)}.
\]
Since
\[
\frac{q_1}{s_1}=\cdots=\frac{q_n}{s_n}=\frac{p}{r},
\]
we have
\begin{align*}
\|\left[Mf\right]^{\frac{p}{r}}\|_{\mathcal{M}_{\vec{s}}^r(\mathbb{R}^n)}
=
\|Mf\|_{\mathcal{M}_{\frac{p}{r}\vec{s}}^{r\frac{p}{r}}(\mathbb{R}^n)}^{\frac{p}{r}}
=
\|Mf\|_{\mathcal{M}_{\vec{q}}^{p}(\mathbb{R}^n)}^{\frac{p}{r}}.
\end{align*}
Thus using Theorem \ref{thm 180122-1}, we obtain
\begin{align*}
\|I_{\alpha}f\|_{\mathcal{M}_{\vec{s}}^r(\mathbb{R}^n)}
&\lesssim \|f\|_{\mathcal{M}_{\vec{q}}^p(\mathbb{R}^n)}^{1-\frac{p}{r}}
\|\left[Mf\right]^{\frac{p}{r}}\|_{\mathcal{M}_{\vec{s}}^r(\mathbb{R}^n)}\\
&=
 \|f\|_{\mathcal{M}_{\vec{q}}^p(\mathbb{R}^n)}^{1-\frac{p}{r}}
\|Mf\|_{\mathcal{M}_{\vec{q}}^p(\mathbb{R}^n)}^{\frac{p}{r}}\\
&\lesssim
 \|f\|_{\mathcal{M}_{\vec{q}}^p(\mathbb{R}^n)}^{1-\frac{p}{r}}
\|f\|_{\mathcal{M}_{\vec{q}}^p(\mathbb{R}^n)}^{\frac{p}{r}}\\
&=\|f\|_{\mathcal{M}_{\vec{q}}^p(\mathbb{R}^n)}.
\end{align*}
\end{proof}

Next, we prove the boundedness of the singular integral operators.
The following theorem seems unknown. Here we include a short proof.
\begin{theorem} \label{thm 180604-1}
Let $1<\vec{q}<\infty$. Then,
\[
\|Tf\|_{\vec{q}}\lesssim\|f\|_{\vec{q}}
\]
for $f \in L^{\vec{q}}(\mathbb{R}^n)$.
\end{theorem}

\begin{proof}
Put $\vec{q}=\theta\vec{r}$, where $\theta>1$ and $\vec{r}>1$.
Then, using the $L^{\vec{r}}(\mathbb{R}^n)$-$L^{\vec{r'}}(\mathbb{R}^n)$ duality argument, for $g\in L^{\vec{r'}}(\mathbb{R}^n)$, we have
\begin{align*}
\|Tf\|_{\vec{q}}=\left\||Tf|^\theta\right\|_{\vec{r}}^\frac1\theta
=\left(\int_{\mathbb{R}^n}|Tf(x)|^\theta g(x){\rm d}x\right)^\frac1\theta.
\end{align*}
Since $g(x)\le M[|g|^\frac1\eta](x)^\eta$ and $M[|g|^\frac1\eta]^\eta \in A_1$ for $\eta>1$, we get
\begin{align*}
\|Tf\|_{\vec{q}}
\le
\left(\int_{\mathbb{R}^n}|Tf(x)|^\theta M[|g|^\frac1\eta](x)^\eta{\rm d}x\right)^\frac1\theta
\lesssim
\left(\int_{\mathbb{R}^n}|f(x)|^\theta M[|g|^\frac1\eta](x)^\eta{\rm d}x\right)^\frac1\theta.
\end{align*}  
By H\"older's inequality and the boundedness of the Hardy--Littlewood maximal operator,
\begin{align*}
\|Tf\|_{\vec{q}}
\lesssim
\left\||f|^\theta\right\|_{\vec{r}}^\frac1\theta\left\| \left(M[|g|^\frac1\eta]\right)^\eta\right\|_{\vec{r'}}
\lesssim
\left\||f|\right\|_{\theta\vec{r}}\left\||g|^\frac1\eta\right\|_{\eta\vec{r'}}^\eta
=\|f\|_{\vec{q}}\|g\|_{\vec{r'}}.
\end{align*}
Thus, the result holds.
\end{proof}

We recall the theorem \ref{thm 180514-1}.

\begin{theorem} 
Let $1<\vec{q}<\infty$ and $1<p<\infty$ satysfy 
\[
\frac np \le\sum_{j=1}^n\frac1q_j.
\]
Then, if we restrict $T$ to $\mathcal{M}_{\vec{q}}^p(\mathbb{R}^n)$, which is intially defined on 
$\mathcal{M}_{\min(q_1, \ldots, q_n)}^p(\mathbb{R}^n)$, 
\begin{equation}\label{eq 180514-2}
\|Tf\|_{\mathcal{M}_{\vec{q}}^p(\mathbb{R}^n)} \lesssim \|f\|_{\mathcal{M}_{\vec{q}}^p(\mathbb{R}^n)} 
\end{equation}
for $f \in \mathcal{M}_{\vec{q}}^p(\mathbb{R}^n)$.
\end{theorem}

\begin{proof}
Let $f \in \mathcal{M}_{\vec{q}}^p(\mathbb{R}^n)$ and $f=f\chi_{2Q}+f\chi_{(2Q)^c} \equiv f_1+f_2$ for any cube 
$Q=Q(z,s)$.
Then, since $T$ is bounded on $L^{\vec{q}}(\mathbb{R}^n)$ by Theorem \ref{thm 180604-1} and 
$f \in L^{\vec{q}}(\mathbb{R}^n)$,
\begin{align} \label{eq 180514-3}
|Q|^{\frac 1p-\frac1n\sum_{j=1}^n\frac1q_j}\|(Tf_1)\chi_Q\|_{\vec{q}}
&\le
|Q|^{\frac 1p-\frac1n\sum_{j=1}^n\frac1q_j}\|Tf_1\|_{\vec{q}} \nonumber \\
&\lesssim
|Q|^{\frac 1p-\frac1n\sum_{j=1}^n\frac1q_j}\|f_1\|_{\vec{q}}
\le
\|f\|_{\mathcal{M}_{\vec{q}}^p(\mathbb{R}^n)}. 
\end{align}
Fix $x \in Q$ and put
\[
f_r(x)=\frac{1}{r^n}\int_{Q(x,r)}|f(y)|{\rm d}y.
\]
Then, by H\"older's inequality, we have
\begin{align*}
|f_r(x)|
&\le \frac{1}{r^n} |Q(x,r)|^{\frac1n\sum_{j=1}^n\frac{1}{q'_j}} \|f\chi_{Q(x,r)}\|_{\vec{q}}\nonumber
\sim  \frac{1}{r^n} r^{\sum_{j=1}^n\frac{1}{q'_j}} \|f\chi_{Q(x,r)}\|_{\vec{q}}\\ 
&=r^{-\frac{n}{p}}r^{\frac{n}{p}-\sum_{j=1}^n\frac{1}{q_j}}\|f\chi_{Q(x,r)}\|_{\vec{q}}
\le r^{-\frac{n}{p}}\|f\|_{\mathcal{M}_{\vec{q}}^p(\mathbb{R}^n)}.
\end{align*}
Thus,
\begin{align}\label{eq 180514-4}
|Tf_2(x)|
&\le
\int_{(2Q)^c}|k(x,y)||f(y)|{\rm d}y
\lesssim
\int_{(2Q)^c}\frac{|f(y)|}{|x-y|^n}{\rm d}y
\lesssim
\int_{2r}^\infty \frac{1}{\ell}f_{\ell}(x) {\rm d}\ell \nonumber\\
&\lesssim
\|f\|_{\mathcal{M}_{\vec{q}}^p(\mathbb{R}^n)}\int_{2r}^\infty \ell^{-\frac{n}{p}-1} {\rm d}\ell
\lesssim
r^{-\frac{n}{p}}\|f\|_{\mathcal{M}_{\vec{q}}^p(\mathbb{R}^n)}.
\end{align}
Thus,  by (\ref{eq 180514-4}), we obtain
\begin{align} \label{eq 180514-5}
|Q|^{\frac 1p-\frac1n\sum_{j=1}^n\frac1q_j}\|(Tf_2)\chi_Q\|_{\vec{q}}
&\lesssim
|Q|^{\frac 1p-\frac1n\sum_{j=1}^n\frac1q_j}r^{-\frac{n}{p}}\|f\|_{\mathcal{M}_{\vec{q}}^p(\mathbb{R}^n)}\|\chi_Q\|_{\vec{q}}
=
\|f\|_{\mathcal{M}_{\vec{q}}^p(\mathbb{R}^n)}.
\end{align}
By (\ref{eq 180514-3}) and (\ref{eq 180514-5}), we get the result.
\end{proof}

\section*{Acknowledgement}

The author would like to thank Professor Yoshihiro Sawano for enthusiastic guidance
and be also thankful to Professor Hitoshi Tanaka
for his kind suggestion on the fractional integral operators.

\end{document}